\documentclass[12pt]{article}
\usepackage{fancyhdr,graphicx}
\usepackage{amsfonts}
\usepackage{amssymb}
\usepackage{amsthm}
\usepackage{newlfont}
\usepackage{mathrsfs}
\usepackage{bbm}
\newtheorem{thm}{Theorem}[section]

\newtheorem{Obser}[thm]{Observation}
\newtheorem{Lemma}[thm]{Lemma}
\newtheorem{cor}[thm]{Corollary}
\newtheorem{pro}[thm]{Proposition}

\usepackage{latexsym,bm}
\title{2-Resonant fullerenes\thanks{This
work is supported by NSFC grant no.10831001} }
\author{Rui Yang, Heping Zhang\footnote{Corresponding author.}
\\\small{School of Mathematics and Statistics, Lanzhou
University, Lanzhou, Gansu 730000, P. R. China}
\\\small{E-mail addresses: yangr2008@lzu.edu.cn, zhanghp@lzu.edu.cn}}

\date{}

\begin{document}

\maketitle \thispagestyle{empty}

\begin{abstract}
 A {\em fullerene graph} $F$ is a planar cubic graph
 with exactly 12 pentagonal faces and other hexagonal faces. A set
 $\mathcal{H}$ of disjoint hexagons of $F$ is called a {\em resonant pattern} (or {\em sextet
 pattern}) if $F$ has a perfect matching $M$ such that every hexagon
 in $\mathcal{H}$ is $M$-alternating. $F$ is said to be $k$-{\em resonant}  if any $i$ ($0\leq i\leq k$) disjoint
hexagons of $F$ form a resonant pattern. It was known that each
fullerene graph is 1-resonant and all 3-resonant fullerenes are only
the nine graphs. In this paper, we show that the fullerene graphs
which do not contain the subgraph $L$ or $R$ as illustrated in Fig.
1 are 2-resonant except for the specific eleven graphs. This result
implies that each IPR fullerene is 2-resonant.

\vskip 0.5cm


 \end{abstract}

%%introduction--------------------------------------------------
\section{Introduction}
A {\em fullerene graph} is a 3-connected planar cubic  graph which has
exactly 12 pentagonal faces and other hexagonal faces. Such graphs
are suitable models for fullerene molecules: carbon atoms are
represented by vertices, whereas edges represent chemical bonds between two atoms. It is well known that a fullerene graph on $n$
vertices exists for any even $n\geq20,n\neq22$ \cite{fowler1995}.
Since the discovery of the first fullerene molecule  C$_{60}$ \cite{kroto1985}
in 1985, the fullerenes have pioneered a new field of study. Various
properties of fullerene graphs were investigated from both chemical
and mathematical points of view
\cite{buhl2001,doslic1998,doslic2002,doslic1999,fowler1995,kardos2008,kardos2009,zhang2001,zhang2007}.

Let $F$ be a fullerene graph. A {\em perfect matching} (or Kekul\'{e}
structure) of $F$ is a set of disjoint edges $M$ such that every
vertex of $F$ is incident with an edge in $M$. It has been shown
\cite{kardos2009} that fullerene graphs have exponentially many
perfect matchings. A set $\mathcal{H}$ of mutually disjoint hexagons
is called a {\em resonant pattern} (or {\em sextet pattern}) if $F$ has
a perfect matching $M$ such that every hexagon in $\mathcal{H}$ is
$M$-alternating. A fullerene graph $F$ is $k$-{\em resonant} (or
$k$-{\em coverable}, $k\geq1$) if any $i$ ($0\leq i\leq k$) disjoint
hexagons of $F$ form a resonant pattern. The concept of resonance
originates from Clar's aromatic sextet theory \cite{clar1972} and
Randi\'{c}'s conjugated circuit model \cite{randic1976,randic1977}.
The $k$-resonance of many types of graphs, including benzenoid graphs, toroidal and Klein-bottle fullerenes,
boron-nitrogen fullerene graphs and (3,6)-fullerene graphs, were
investigated extensively
\cite{chen1993,Li2011,yang2011,ye2009,zhang2004,zhang2010,zheng1991}.

In \cite{ye2009} Ye et al. showed that every
fullerene graph is 1-resonant and there are exactly nine
fullerene graphs
$F_{20},F_{24},F_{28},F_{32},F_{36}^{1},F_{36}^{2},F_{40},F_{48},F_{60}$
which are also $k$-resonant for each $k\geq3$, but not all
fullerene graphs are 2-resonant. They also proved that every leapfrog fullerene graph is 2-resonant and asked a problem: whether a fullerene graph satisfying the isolated pentagon rule (IPR) is 2-resonant.  In \cite{kaiser2010}, Kaiser et al.
gave a positive answer to the problem.

\begin{figure}[h]
\begin{center}
\includegraphics[scale=0.6
]{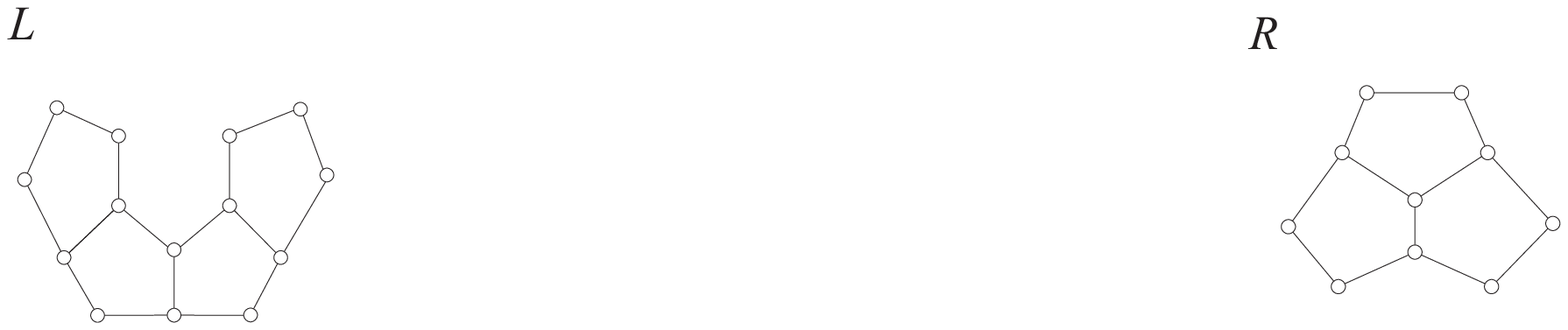}\\{Figure 1.  The subgraphs $L,R$.}
\end{center}
\end{figure}

In this paper, we consider fullerene graphs which are allowed to have some adjacent pentagons, i.e. violating IPR.
Two substructures $L$ and $R$ consisting of four and three pentagons are defined  in  Fig.  1. We characterize fullerene graphs without substructures $L$ and $R$ which are 2-resonant and obtain the following main theorem.

\begin{thm}\label{lr} Let $F$ be a fullerene graph which does not
contain the subgraph $L$ or $R$. Then
$F$ is 2-resonant except for the eleven fullerene graphs in Fig.  2.
\end{thm}
\begin{figure}[h]
\begin{center}
\includegraphics[scale=0.75
]{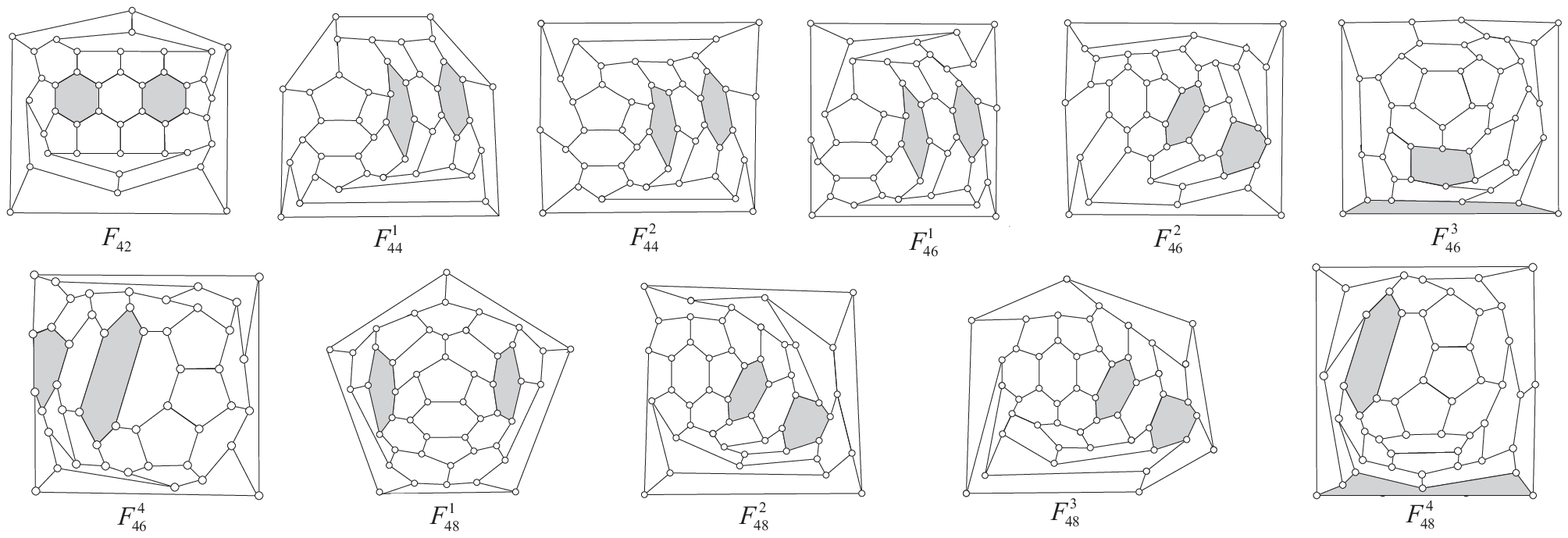}\\{Figure 2.  The eleven non-2-resonant fullerene graphs  without   subgraph $L$ or $R$.}
\end{center}\label{5}
\end{figure}

It is easy to verify that the eleven fullerene graphs depicted in Fig.  2 are not 2-resonant since the two grey hexagons do not form a resonant pattern.

A  fullerene graph is said to be IPR if it satisfies the isolated pentagon rule (IPR) (i.e. any pentagons are disjoint).  It is obvious that every IPR fullerene graph has no substructures $L$ or $R$ and each graph in Fig.  2 has at least a pair of adjacent pentagons. So  Theorem \ref{lr} implies immediately the following known result.

\begin{cor}\label{IPR} \cite{kaiser2010} Every IPR fullerene graph is 2-resonant.
\end{cor}

\section{Definitions and preliminaries}
Let $G$ be a connected plane graph with vertex-set $V(G)$ and
edge-set $E(G)$. For $X, Y\subset V(G)$, we define $E(X,Y)$ the set
of edges having one end-vertices  in $X$ and the other in $Y$. We simply
write $\bigtriangledown(X)$ for $E(X,\overline{X})$ where
$\overline{X}= V(G)-X$. For  subgraphs $H$ and $H'$ of $G$, we also simply write $\bigtriangledown (H)$ for $\bigtriangledown (V(H))$, and $E(H,H')$ for $E(V(H),V(H'))$;  We call $H$ is {\em incident} to $H'$ if $V(H)\cap
V(H')=\emptyset$ and $E(H,H')\neq \emptyset$. For a
face $f$ of $G$, its boundary is a closed walk and we often
represent it by its boundary if unconfused. Pentagonal and hexagonal faces are referred to simply as pentagons
and hexagons. Use $\partial(G)$ to denote the boundary of
the exterior face of $G$.

A graph $G$ is called {\em factor-critical} if
$G-v$ has a perfect matching for every vertex $v\in V(G)$. A
factor-critical graph is {\em trivial} if it consists of a single
vertex.

\begin{Obser}\cite{kaiser2010}\label{2connected} Every non-trivial factor-critical
subgraph of a fullerene graph is 2-connected.
\end{Obser}

We call a vertex set $S\subseteq V(G)$ {\em matchable} to
$\mathcal{C}_{G-S}$ if the (bipartite) graph $G_{S}$, which arises
from $G$ by contracting the components $C \in \mathcal{C}_{G-S}$ to
single vertices and deleting all the edges inside $S$, contains a
matching of $S$. The following critical theorem (\cite{diestel2006},
Theorem 2.2.3) may be viewed as a strengthening of Tutte's 1-factor
theorem \cite{tutte1947}.

\begin{thm}\label{matchable} Every graph $G$ contains a vertex set
$S$ with the following two properties:

(1) $S$ is matchable to $\mathcal{C}_{G-S}$;

(2) Every component of $G-S$ is factor-critical.

Given any such set $S$, the graph $G$ contains a perfect matching if
and only if $|S|=|\mathcal{C}_{G-S}|$.

\end{thm}

An {\em edge-cut} of a connected graph $G$ is a set of edges $C\subset
E(G)$ such that $G- C$ is disconnected. An edge-cut $C$ of $G$ is
{\em cyclic} if each component of $G- C$ contains a cycle. A graph $G$
is {\em cyclically k-edge-connected} if $G$ cannot be separated
into two components, each containing a cycle, by removing less than
$k$ edges. The {\em cyclical edge-connectivity} of $G$, denote by
$c\lambda(G)$, is the greatest integer $k$ such that $G$ is cyclically
$k$-edge-connected. For a fullerene graph $F$, T. Do\v{s}li\'{c}
\cite{doslic2003}, and Qi and Zhang \cite{Qi2008} proved that $c\lambda(F)=5$; F. Kardo\v{s} and R. \v{S}krekovski \cite{kardos2008} obtained the same result by  three operations on cyclic edge-cuts. There are at least
twelve cyclic 5-edge-cuts---formed by the edges pointing outward
each pentagonal face and there are also cyclic 6-edge-cuts---formed
by the edges pointing outward each hexagonal face. These cyclic
5- and 6-edge-cuts are called {\em trivial}. A cyclic edge-cut $C$ of a
fullerene graph $F$ is {\em non-degenerated} if both components of
$F-C$ contain precisely six pentagons. Otherwise, $C$ is
{\em degenerated}. Obviously, the trivial cyclic edge-cuts are
degenerated.

F. Kardo\v{s} and R. \v{S}krekovski \cite{kardos2008}, and K. Kutnar and D. Maru\v{s}i\v{c} \cite{Kutnar2008} independently gave the nanotube structure of fullerene graphs with a non-trivial 5-cyclic edge-cut.

\begin{thm}\label{cyclic5}\cite{kardos2008,Kutnar2008}  A fullerene graph has non-trivial
cyclic 5-edge-cuts if and only if it is isomorphic to the graph
$G_{k}$ for some integer $k\geq 1$, where $G_{k}$ is the fullerene
graph comprised of two caps formed of six pentagons joined by k
layers of hexagons (see Fig.  3).
\end{thm}

\begin{figure}[h]
\begin{center}
\includegraphics[scale=0.5
]{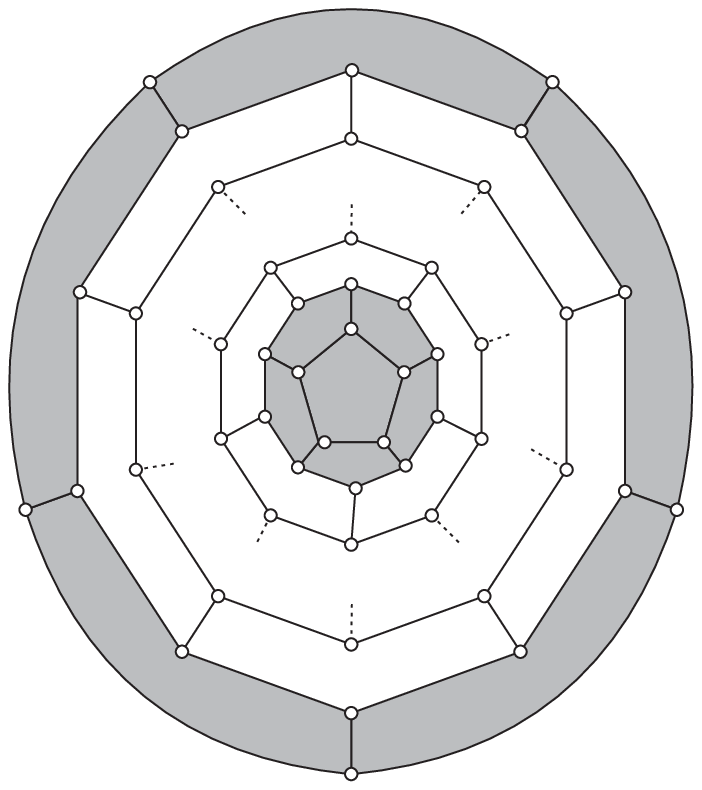}\\{Figure 3.  The graph $G_{k}$ are the only fullerene
graphs with non-trivial cyclic 5-edge-cuts.}\label{gk}
\end{center}
\end{figure}

Further F. Kardo\v{s} and R. \v{S}krekovski  listed
the degenerated cyclic 6-edge-cuts in fullerene graphs.

\begin{figure}[h]
\begin{center}
\includegraphics[scale=0.7
]{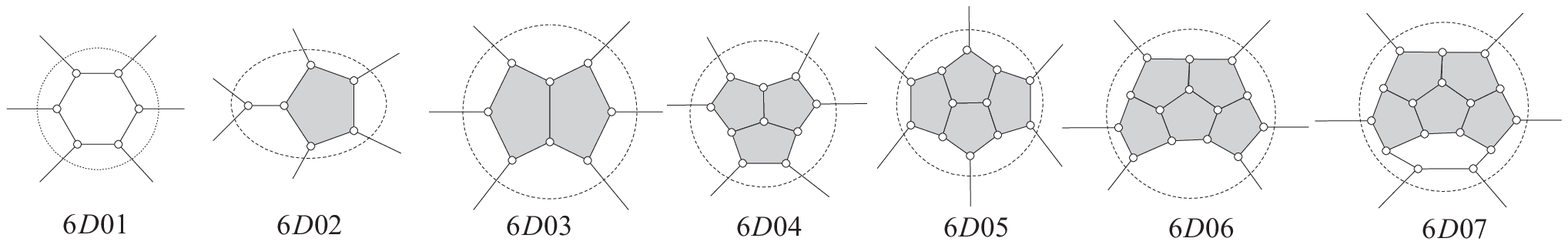}\\{Figure 4.  Degenerated cyclic
6-edge-cuts.}
\end{center}\label{degeneratedcyclic6edgecuts}
\end{figure}
 \begin{thm}\cite{kardos2008}\label{cyclic6}  There are precisely seven
 non-isomorphic graphs that can be obtained as components of
 degenerated cyclic 6-edge-cuts with less than six pentagons (see
 Fig.  4).

 \end{thm}

Recently, F. Kardo\v{s} et al. characterized the
degenerated cyclic 7-edge-cuts in fullerene graphs.

\begin{thm}\cite{kardos2010}\label{cyclic7} There exist 57 non-isomorphic graphs that can be
obtained as components of degenerated cyclic 7-edge-cuts with less
than six pentagons (see Fig.  5).
\end{thm}

\begin{figure}[h]
\begin{center}
\includegraphics[scale=0.7
]{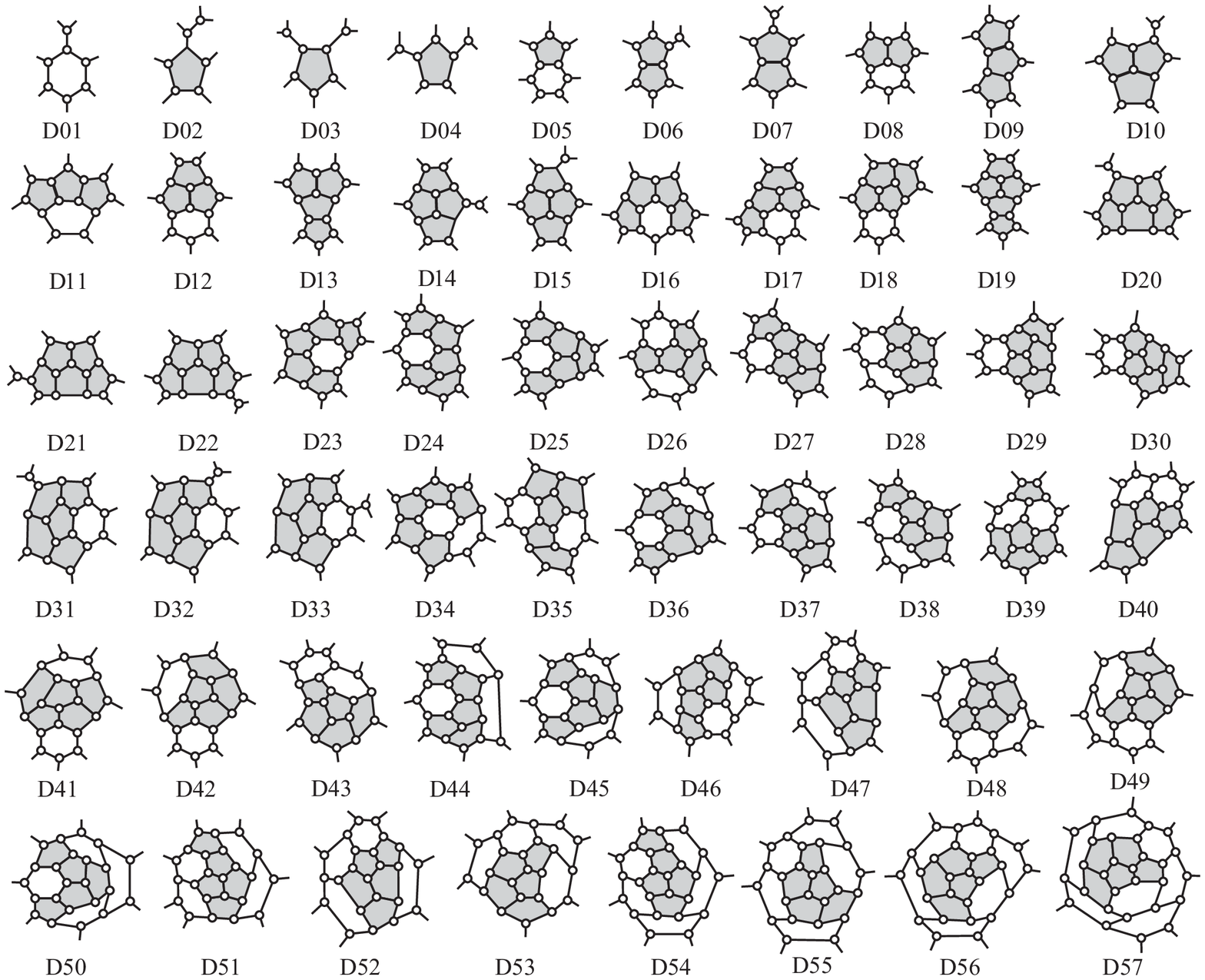}\\{Figure 5. Degenerated cyclic
7-edge-cuts.}
\end{center}
\end{figure}

Their characterizations are  based on the following result.

\begin{thm}\cite[Theorem 1]{kardos2008}\label{obtainedbytrivialones} The cyclic edge-cuts of a fullerene graph can be
constructed from the trivial ones using the reverse operations of
$(O_{1}),(O_{2})$ and $(O_{3})$.
\end{thm}

Here the three operations can be presented as follows (see Fig.  6
for an illustration).

$(O_{1})$ If a component $H$ contains a vertex of degree one, then
using $(O_{1})$ one can modify the $k$-edge-cut $C$ into a
$(k-1)$-edge-cut $C_{1}$.

$(O_{2})$ If a component $H$ contains two adjacent vertices of degree
two, then using $(O_{2})$ one can modify the $k$-edge-cut $C$ into a
$k$-edge-cut $C_{2}$.

$(O_{3})$ If the vertices of the outer faces of $H$ are consecutively
of degree 2 and 3, then using $(O_{3})$ one can modify the
$k$-edge-cut $C$ into a $k$-edge-cut $C_{3}$.

\begin{figure}[h]
\begin{center}
\includegraphics[scale=0.7
]{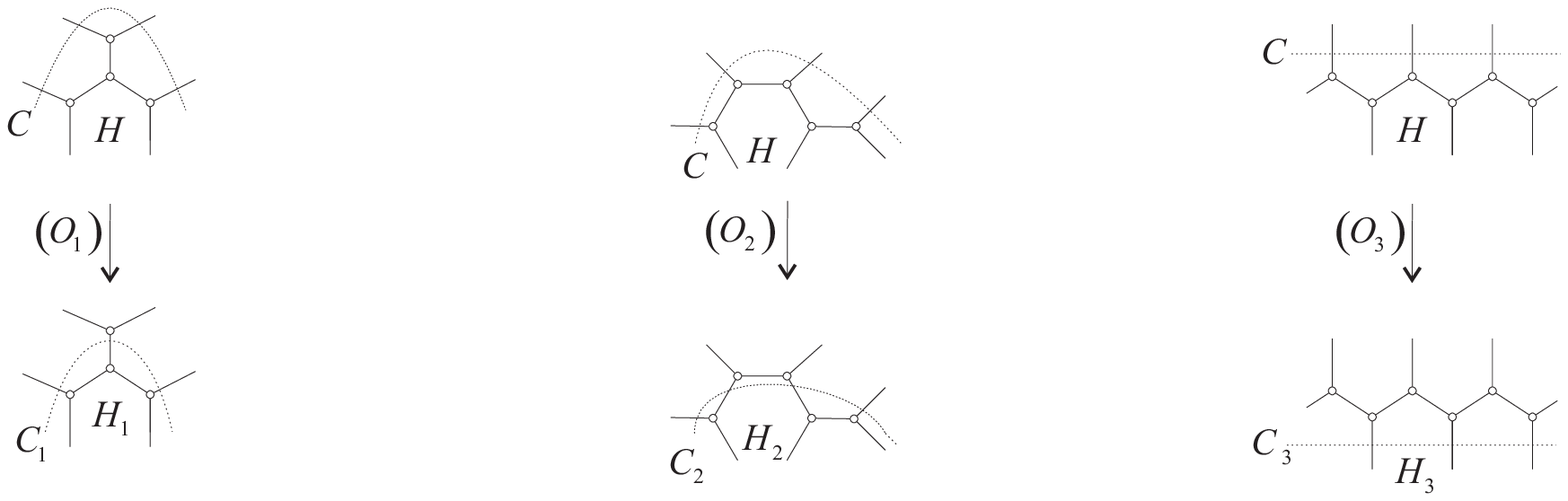}\\{Figure 6.  Three operations $O_{1},O_{2}$ and $O_{3}$.}
\end{center}\label{operations}
\end{figure}

Let $G$ be a subgraph of a fullerene graph $F$. A face $f$ of $F$ is
a {\em neighboring face} of $G$ if $f$ is not a face of $G$ and $f$ has
at least one edge in common with $G$. For two faces $f_{1},f_{2}$ of
a fullerene graph $F$, we always say $f_{1}$ {\em intersects} $f_{2}$ if
$f_{1}$ is a neighboring face of $f_{2}$. The following result is
known.

\begin{Lemma}(\cite{zhang2010}, Lemma 4.2)\label{intersectatoneedge} Let $H$ be a 3-regular plane graph. If $H$ is cyclically 4-edge-connected, then
there are neither three faces which are pairwise adjacent but do not
share a common vertex, nor two faces which share more than one
disjoint edges in $H$.
\end{Lemma}

\begin{Lemma}\label{atmostoneneighboring} Let $f$ and $f'$ be two disjoint faces of a fullerene graph
$F$ with
$E(f,f')=\emptyset$. Then there is at most one common neighboring
face of both $f$ and $f'$.
\end{Lemma}

\begin{proof} To the contrary, suppose at least two such neighboring
faces exist, say $f_{1}$, $f_{2}$. Then the edge set $C=\{e_{1},e_{2},e_{3},e_{4}\}$ forms an edge cut, where $e_{1}=\partial(f)\cap\partial(f_{1}), e_{2}=\partial(f)\cap\partial(f_{2}), e_{3}=\partial(f')\cap\partial(f_{2}), e_{4}=\partial(f')\cap\partial(f_{1})$. Since $E(f,f')=\emptyset$, $f_{1},f_{2}$ are hexagons. On the other hand, if one component $H$ of $F-C$ contains a vertex $v$ of degree one in $H$, then exactly two edges incident with $v$ belong to the cut $C$, say $e_{1},e_{2}$. Let the third edge incident with $v$ be $e$. Then $C_{1}=C\setminus \{e_{1},e_{2}\}\cup\{e\}$ is a cyclic 3-edge-cut in $F$ since $E(f,f')=\emptyset$ (see Fig.  7(a)), contradicting that
$c\lambda(F)=5$. If both components of $F-C$ are of minimum degree two, then $C$ is a cyclic 4-edge-cut in $F$ (see Fig.  7(b)), again contradicting that
$c\lambda(F)=5$.
\end{proof}

\begin{figure}[h]
\begin{center}
\includegraphics[scale=0.9
]{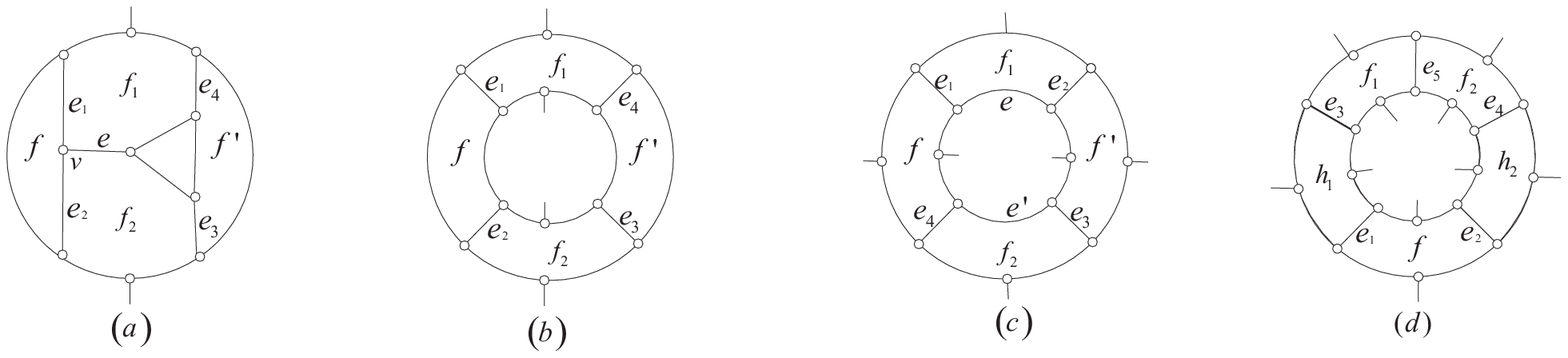}\\{Figure 7.  (a) Cyclic 3-edge-cut
$\{e,e_{3},e_{4}\}$, (b) and (c) two cyclic 4-edge-cuts
$\{e_{1},e_{2},e_{3},e_{4}\}$, and (d)  cyclic 5-edge-cut
$\{e_{1},e_{2},e_{3},e_{4},e_{5}\}$.}
\end{center}\label{nocyclic4}
\end{figure}

\begin{Lemma}\label{lessthan1} Let $f, f'$ be two faces of a fullerene graph
$F$. Then $|\nabla(f) \cap \nabla(f')|\leq 1$.
\end{Lemma}
\begin{proof} Suppose to the contrary that there exist two edges
$e,e'$ of $F$ satisfying $e,e'\in \nabla(f) \cap \nabla(f')$.
Then both $e$ and $e'$ are contained in a neighboring face of $f$
and $f'$, say $f_{1},f_{2}$, respectively. Again a cyclic edge-cut
of size less than five can be gained (see Fig.  7(c)), also a
contradiction.
\end{proof}

\begin{Lemma}\label{equalemptyset} Assume that a fullerene graph $F$ has no non-trivial cyclic 5-egde-cut,  disjoint faces $h_{1}$ and
 $h_{2}$ of  $F$ are not incident and share a common neighboring face $f$. Let $f_{1},f_{2}$ be the other neighboring faces of
 $h_{1},h_{2}$ (respectively) different from $f$ and $V(f_{i})\cap V(f)=\emptyset$ for some $i\in \{1,2\}$. Then $V(f_{1})\cap
 V(f_{2})=\emptyset$.
\end{Lemma}
\begin{proof} Let $E(h_{1})\cap E(f)=\{e_{1}\}$,$E(h_{2})\cap
E(f)=\{e_{2}\}$. Since $h_{1}$ and $h_{2}$ are not incident, $f$ is
a hexagon. By Lemma \ref{intersectatoneedge} we can know that
$f_{1}$ intersects $h_{1}$ at exactly one edge, say $E(h_{1})\cap
E(f_{1})=\{e_{3}\}$. Similarly, $E(h_{2})\cap E(f_{2})=\{e_{4}\}$.
Moreover, $f_{2}$ ($f_{1}$) can not intersect $h_{1}$ ($h_{2}$) by
Lemma \ref{atmostoneneighboring}. Suppose $V(f_{1})\cap
V(f_{2})\neq\emptyset$. Then again by Lemma \ref{intersectatoneedge}
we can set $\{e_{5}\}=E(f_{1})\cap E(f_{2})$. Since $V(f_{i})\cap
V(f)=\emptyset$ for $i\in \{1,2\}$, $C=\{e_{1},\cdots,e_{5}\}$ forms
a cyclic 5-edge-cut (see Fig.  7(d)). Thus it is a trivial one,
contradicting that $h_{1}$ and $h_{2}$ are not incident.\end{proof}

A {\em fragment} $B$ of a fullerene graph $F$ is a subgraph of $F$
consisting of a cycle together with its interior. A {\em pentagonal
fragment} is a fragment with only pentagonal inner faces. For a
fragment $B$, all 2-degree vertices of $B$ lie on its boundary. A path $P$ on $\partial(B)$ connecting two 2-degree vertices is {\em degree-saturated} if $P$ contains no 2-degree vertices of $B$ as intermediate vertices.

\begin{Lemma}(\cite{dong2009}, Lemma 2.2)\label{aforest} Let $B$ be a fragment of a fullerene
graph $F$ and $W$ the vertex set consisting of all 2-degree vertices
on $\partial (B)$. If $0< |W| \leq 4$, then $T= F-(V(B)\setminus W)$ is a
forest and

(1) $T$ is $K_{2}$ if $|W|=2$;

(2) $T$ is $K_{1,3}$ if $|W|=3$;

(3) $T$ is either the union of two $K_{2}^{'}s$, or a 3-length path, or
$T_{0}$ as shown in Fig.  8 if $|W|=4$.
\end{Lemma}

\begin{figure}[h]
\begin{center}
\includegraphics[scale=0.7
]{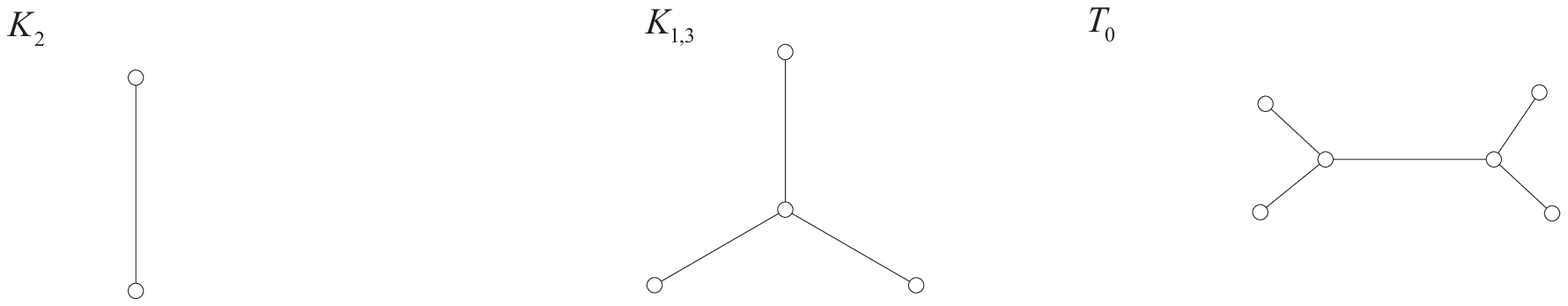}\\{Figure 8.  Trees: $K_{2},K_{1,3}$ and $T_{0}$.}
\end{center}\label{4}
\end{figure}

\section{Proof of Theorem \ref{lr}}

Let $F$ be a fullerene graph without $L$ or $R$ that is different from anyone of graphs in Fig.  2. To the  contrary, suppose $F$ is not 2-resonant. That is,  there exist two disjoint hexagons $h_{1},h_{2}$ such that $F-V(h_{1} \cup
h_{2})$ does not have a perfect matching. Then Theorem
\ref{matchable} guarantees the existence of a vertex set $A$ of
$F-V(h_{1} \cup h_{2})$ such that every component of $F-(V(h_{1}
\cup h_{2})\cup A)$ is factor-critical and the number of these
factor-critical components is more than $|A|$. We make the following notations.

$H=V(h_{1} \cup h_{2})$,

$\mathcal{D}$:  the collection of
factor-critical components of $F-H-A$, and

 $\mathcal{D}^{*}$: the collection of
non-trivial factor-critical components of  $F-H-A$.

For convenience, let $D$ and $D^*$ denote respectively the union of vertex-sets of components in $\mathcal{D}$ and $\mathcal{D}^{*}$,
and $D_{0}=D-D^{*}$; let $E(\cal D^*)$ denote the union of edge-sets of components in $\mathcal{D}^{*}$. Then $D_0$ is an independent set of $F$.
 In Fig. 9 we divide $V(F)$ into $H,A,D^{*},D_{0}$.

\begin{figure}[h]
\begin{center}
\includegraphics[scale=0.7
]{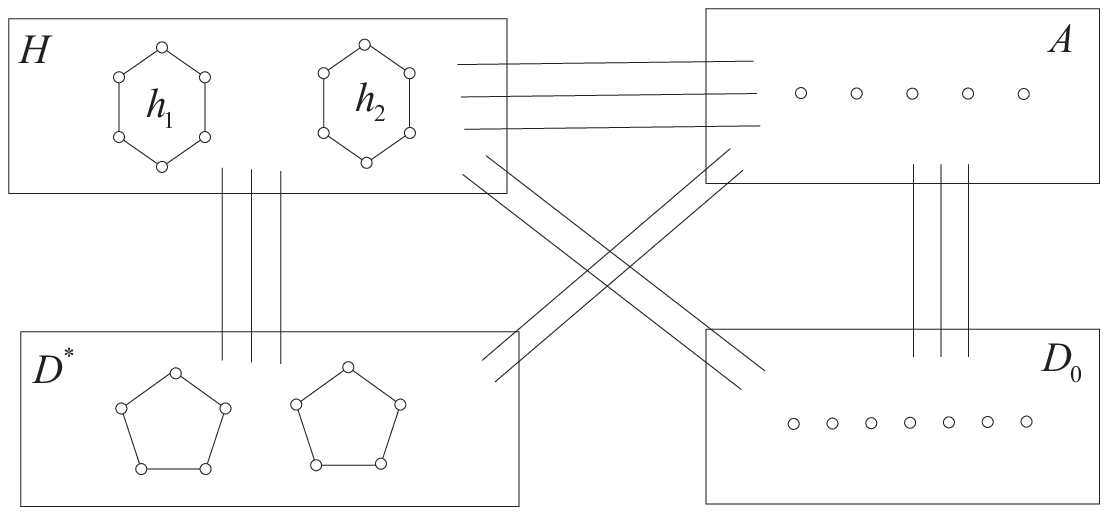}\\{Figure 9.  The partition of $V(F)$ into $H,A,D^{*}$
and $D_{0}$.}
\end{center}
\end{figure}

In what follows,
we will show that the components in $\mathcal{D}$ are singletons. That is, $D^*=\emptyset$, and
 the vertices of $F$ can be reclassified  into $H, A$ and $D_{0}$. Finally,
by means of the structures of the neighboring faces of $h_{1},h_{2}$
and the fact that $F$ contains no subgraph $L$ or $R$ we can
construct the fullerene graphs satisfying the conditions. In our
constructing process we get the contradictions.

Since $|\mathcal{D}|$ and $|A|$ have the same parity and
$|\mathcal{D}|> |A|$, we have

\hspace{4cm} $|\mathcal{D}|\geq |A|+2$. \hfill(1)

Furthermore, $A\cup H$ sends out exactly $|\nabla(A\cup H)|$ edges,
and

\hspace{1cm} $|\nabla(A\cup
H)|=|\nabla(H)|+3|A|-2|E(A,H)|-2|E(A,A)|$. \hfill(2)

Although there is no even components in $\mathcal{D}$,
here we still use the notation system in  the proof of Kaiser's \cite{kaiser2010}.
Let  $$s(\mathcal{D})=\sum_{F^{*}\in\mathcal{D}}\frac{|\nabla(F^*)|-3}{2},$$ where each term is non-negative since $F$ is 3-connected.
Then $D$ sends out precisely $\nabla(D)$ edges, and

\hspace{20mm}$|\nabla(D)|$=$\sum_{F^{*}\in\mathcal{D}}|\nabla(F^*)|$=$3|\mathcal{D}|+2s(\mathcal{D})$
\hfill(3)

As $|\nabla(A\cup H)|$=$|\nabla(D)|$, (1), (2) and (3) imply that

\hspace{20mm}$|E(A,H)|+|E(A,A)|+s(\mathcal{D})\leq\frac{1}{2}|\nabla(H)|-3$.
\hfill(4)

 $|\nabla(H)|$ equals 12 or 10 by Lemma \ref{lessthan1}. From (4)
we have

\hspace{1mm}$|E(A,H)|+|E(A,A)|+s(\mathcal{D})\leq3$ if
$h_{1}$ and $h_{2}$ are not incident, and\hfill(5)

\hspace{1mm}$|E(A,H)|+|E(A,A)|+s(\mathcal{D})\leq2$ if
$h_{1}$ and $h_{2}$ are incident.\hfill(6)

Let $\mathcal{P}$ be the set of pentagons of $F$. Then
$|\mathcal{P}|=12$. If $X\subseteq V(F)$ and $Y\subseteq E(F)$, let
$p(X)=|\{P\in \mathcal{P}|V(P)\cap X\neq\emptyset\}|$ and
$p(Y)=|\{P\in \mathcal{P}|E(P)\cap Y\neq\emptyset\}|$. Observe that
 $F-H-D^{*}-E(A,A)$ is a bipartite graph with bipartition $D_0\cup A$. That is,

\emph{all twelve pentagons of $F$ must contain a vertex in $H\cup
D^{*}$ or an edge in $E(A,A)$.}\hfill(*)

In particular, $p(H)+p(E(A,A))+p(D^{*})\geq12$  \hfill(7)

Since $F$ contains no $L,R$ as subgraphs, $p(V(h_{1}))\leq4$,
$p(V(h_{2}))\leq4$, and $p(V(F^*))\leq3$ if $F^{*}\in \cal D^*$ is a
pentagon. \hfill(8)

\begin{Lemma}\label{trivial}$F$ has no non-trivial cyclic 5-edge cuts.
\end {Lemma}
\begin{proof}
 For a cyclic 5-edge-cut  of $F$, it must be a
 {\em trivial} one; otherwise, $F$ would be isomorphic to $G_{k}$
for some integer $k\geq1$ by Theorem \ref{cyclic5}, which has the subgraph $R$, a contradiction.
\end{proof}

By  Observation \ref{2connected} and Lemma \ref{trivial},  we have

\begin{pro}\label{5edgecut} A  component $F^{*}$ in  $\cal D$ with $|\nabla(F^{*})|=5$ is a pentagon.
\end{pro}

For convenience, in some of the following  figures, the
black vertices and the crossed vertices always represent the
vertices belonging to $A$ and $D_{0}$ respectively,   the
non-trivial factor-critical components  are drawn with black lines
and the grey hexagons refer to $h_{1}$ and $h_{2}$.

For a face $f$ of $F$, we call an edge $e$ on $\partial(f)$ a {\em
contributing edge} if it belongs to $E(A,H)$,  $E(A,A)$,
$E(\mathcal{D}^{*})$, or $E(h_{1},h_{2})$. More precisely,
edges in $E(A,H)$, $E(A,A)$, $E(\mathcal{D}^{*})$,
$E(h_{1},h_{2})$ are sometimes called the $E(A,H)$, $E(A,A)$,
$E(\mathcal{D}^{*})$, $E(h_{1},h_{2})$ edges, respectively. Lemma \ref{lessthan1} implies $|E(h_{1},h_{2})|\leq1$. By
 Ineqs. (5) and (6), $|E(A,H)|\leq3$, $|E(A,A)|\leq3$,
 and $s(\mathcal{D})\leq3$. The latter implies
$|\nabla(F^{*})| \leq9$ for each $F^{*}\in \cal D$.

 For $F^*\in \cal D^*$,  $F^*$ is a 2-connected factor-critical graph that is a subgraph of $F$.  If $|\nabla(F^{*})| \leq 7$, $F^*$ has exactly one face that is not a face of $F$ by Lemma \ref{aforest}. Hence $F^*$ can be viewed as a fragment of $F$. If $|\nabla(F^{*})| =9$, by the similar reason together with $c\lambda(F)=5$ we have that $F^*$ has at most two faces that are not  faces of $F$ and both $h_1$ and $h_2$ lie in the same  such face of $F^*$. So we make a convention: both $h_1$ and $h_2$ lie in the exterior face of $F^*$.

We now give some
characterizations about the faces in $F$.

\begin{Lemma}\label{intersectingboth} For $j\in \{1,2\}$, if a
neighboring face $f$ of $h_{j}$ includes no contributing edges, then
$f$ is either a pentagon with the boundary $HHD_{0}AD_{0}$ (which
means the vertices on $\partial(f)$ are consecutively in
$H,H,D_{0},A,D_{0}$. The following notations have the same sense) or
a hexagon with the boundary
$HHD_{0}HHD_{0}$ which is  adjacent to both $h_{1}$ and $h_{2}$.
\end{Lemma}

The lemma  can be proved in  the same way as Lemma 9 in \cite{kaiser2010}.

\begin{cor}\label{atleastoneedge} At least one neighboring face of
$h_{j}$ contains a contributing edge for $j\in \{1,2\}$.
\end{cor}
\begin{proof} To the contrary, suppose every  neighboring face of
$h_{j}$ has no contributing edges. By Lemma \ref{intersectingboth} the
neighboring hexagonal faces of $h_{j}$ intersect  $h_{3-j}$.
Using Lemma \ref{atmostoneneighboring} we can know at most one
neighboring face of $h_{j}$ is hexagonal. Then its remaining
neighboring pentagonal faces form a subgraph  of $F$ containing $L$,
contradicting the assumption.
\end{proof}

Lemma \ref{intersectingboth} can be generalized as the following
results.

\begin{Lemma}\label{characterizationh} For  a
neighboring face $f$ of $h_{j}$,  $j\in \{1,2\}$, we have the following results.

\hspace{1mm}(1) If $f$ contains precisely one contributing edge, which
belongs to $E(\mathcal{D}^{*})$,  $E(A,H)$ and $E(A,A)$, respectively, then $f$
is a hexagon with the boundaries
$HHD^{*}D^{*}AD_{0},HHAD_{0}AD_{0}$ and $HHD_{0}AAD_{0}$.

\hspace{1mm}(2) If $f$ contains precisely one contributing edge, which belongs to $E(h_{1},h_{2})$, then $f$ is a pentagon with the
boundary $HHHHD_{0}$.

\hspace{1mm}(3) $f$ cannot contain two $E(A,A)$ edges.

\hspace{1mm}(4) $f$ cannot contain one $E(A,A)$, one
$E(\mathcal{D}^{*})$ and one $E(A,H)$ edge.

\hspace{1mm}(5) $f$ cannot contain two $E(A,H)$ edges and one
$E(\mathcal{D}^{*})$ edge.
\end{Lemma}

\begin{proof} Let $ab$ be a common edge of $h_{j}$ and $f$,
$a',b'$  the neighbors of $a,b$, respectively, not on $h_{j}$,
 $x$ the neighbor of $a'$ on $\partial(f)$ but different
from $a$ (see Fig. 10(a)).

\begin{figure}[h]
\begin{center}
\includegraphics[scale=0.7
]{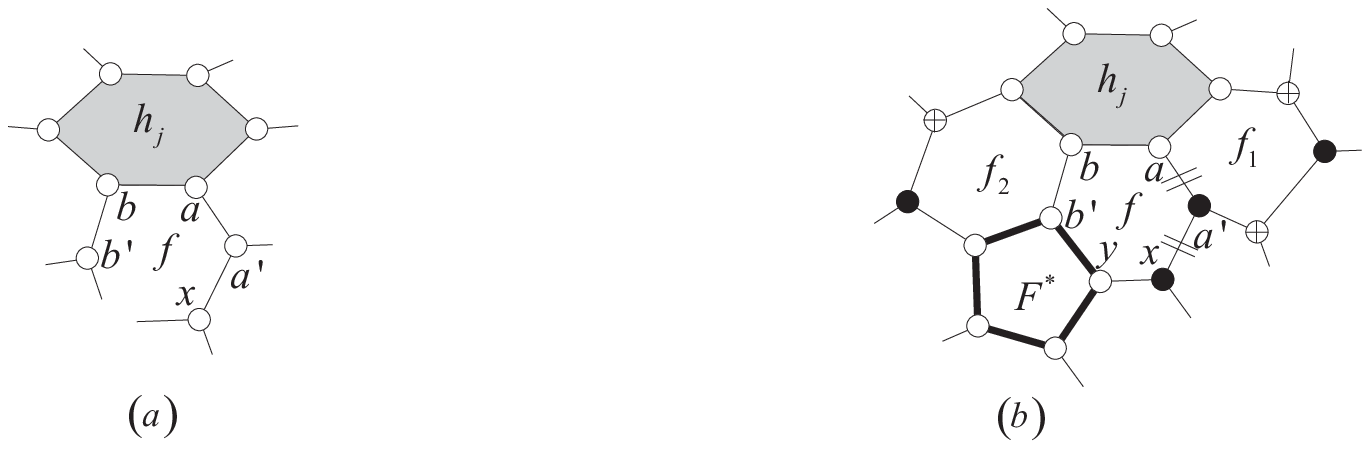}\\{Figure 10. (a) The labellings of face $f$, (b)
Illustration to the proof of Lemma \ref{characterizationh}(4).}
\end{center}
\end{figure}

(1) If  exactly one contributing edge of $f$ belongs to
$E(\mathcal{D}^{*})$, then $a',b'\in D^{*}\cup D_{0}$. So one of $a',b'$, say $a'$, belongs to $D^{*}$. Otherwise,  there would exist edges between $D_{0}$
and $D^{*}$,  a contradiction.
Further $x\in
D^{*}$ by Observation \ref{2connected}. Now $b'\in D_{0}$ as $a'x\in
E(\mathcal{D}^{*})$ being a unique contributing edge of $f$ . Moreover, $b'$ and $x$ must have a common
neighbor belonging to $A$ on $\partial(f)$ since there are no edges
between $D_{0}$ and $D^{*}$. Thus $f$ is a hexagon with the boundary
$HHD^{*}D^{*}AD_{0}$. If $f$ contains exactly one $E(A,H)$ edge,
then we have that  $a',b'\in A\cup D_0$ in an analogous way. Further, we have $a'\in A$ and $b'\in D_{0}$, say.
Also $x\in D_{0}$. Again $b'$ and $x$ must share a common neighbor
belonging to $A$ on $\partial(f)$. Then   $f$ is a hexagon with the
boundary $HHAD_{0}AD_{0}$. If $f$ contains exactly one $E(A,A)$
edge, then $a',b'\in D_{0}$ and $f$ is a hexagon with the boundary
$HHD_{0}AAD_{0}$.

(2) Since $f$ contains precisely one contributing edge, that  belongs to
$E(h_{1},h_{2})$, one of $a',b'$, say $a'$, belongs to
$V(h_{2})$, the other $b'$ to $D_{0}$. Then also $x\in V(h_{2})$ by the
3-regularity of $F$. Now $b'$ and $x$ must be adjacent since $aa'$ is only
one contributing edge of $f$. Hence  $f$
is a pentagon with the boundary $HHHHD_{0}$.

(3) Suppose $f$ contains two $E(A,A)$ edges $e_1$ and $e_2$. Then $f$ includes at least one
$E(A,H)$ edge. Ineqs. (5) and (6) imply that $f$ is a hexagon, $|E(A,H)|=1$,
 $|E(A,A)|=2$, and $s(\cal D)=0$. The latter together with $c\lambda(F)=5$ imply that each component in $\cal D$ is a single vertex; that is $D^*=\emptyset$.
 Hence  $p(H)+p(E(A,A))+p(D^{*})\leq
(p(V(h_{1}))+p(V(h_{2})))+p(\{e_{1},e_{2}\})+0\leq (4+4)+2=10$,
contradicting Ineq. (7).

(4) Suppose $f$ contains one $E(A,A)$, one $E(\mathcal{D}^{*})$ and
one $E(A,H)$ edges. Then $f$ is a hexagon with the boundary
$HHAAD^{*}D^{*}$. Let $aa'xyb'ba$ be the boundary of $f$ along
clockwise direction. Without loss of generality, we may assume
$a',x\in A,b',y\in D^{*}$. Then by Ineqs. (5) and (6) we can know
that $|E(A,H)|=1, |E(A,A)|=1$ and $s(\mathcal{D})=1$. The latter
together with Prop. \ref{5edgecut} imply  $|\cal D^*|=1$ and the
component in  $\cal D^*$ is a pentagon. Let  $f_{1}$ and $f_{2}$ be
the common neighboring faces of $h_{j}$ and $f$, $F^{*}$ the
pentagonal
 component in $\cal D^*$ (see Fig.  10(b)). Then both $f_{1}$
and $f_{2}$ are hexagons by (1). Moreover,
$p(D^{*})=p(V(F^{*}))\leq3$ by Ineq. (8). So
$p(H)+p(E(A,A))+p(D^{*})\leq
(p(V(h_{1}))+p(V(h_{2})))+p(\{a'x\})+p(V(F^{*}))\leq (4+3)+1+3=11$,
contradicting Ineq. (7).

(5) Suppose $f$ contains two $E(A,H)$ edges and one
$E(\mathcal{D}^{*})$  edge. In an analogous way  as (4), we show that $f$ is a hexagon and the
 component in $\cal D^*$, say $F^{*}$, is a pentagon. Moreover,
$p(E(A,A))=0,p(D^{*})=p(V(F^{*}))\leq3$. Again we have
$p(H)+p(E(A,A))+p(D^{*})\leq(4+4)+3=11$, a contradiction.
\end{proof}

Note: By Observation \ref{2connected} and Lemma
\ref{characterizationh} (5) we can know if $h_{i}$ is not incident to
a non-trivial factor-critical component $F^{*}$, then the
neighboring faces of $h_{i}$ do not contain $E(F^{*})$ edges for
$i\in \{1,2\}$. This fact is used elsewhere in this paper.

\begin{Obser}\label{characterizationm2} Let $f$ be a face of $F$
with precisely one contributing edge, which belongs to $E(A,A)$.
Then $f$ is either a pentagon with the boundary $AAD_{0}AD_{0}$ or a
hexagon with the boundary $HHD_{0}AAD_{0}$.

\end{Obser}
\begin{proof} Let $ab\in E(A,A)\cap E(f)$ and $a',b'$ be the neighbors
of $a,b$, respectively, on $\partial(f)$ but $a'\neq b,b'\neq a$.
Then $a',b'\in D_{0}$ as $f$ contains precisely one contributing
edge $ab$. If $f$ is a pentagon, then $a'$ and $b'$ must have a
common neighbor on $\partial(f)$ belonging to $A$ and $f$ has the
form $AAD_{0}AD_{0}$. If $f$ is a hexagon $aa'xyb'ba$, then
$x,y\notin D$. Since $f$ has no other $E(A,A)$ edge except for $ab$,
one of $x$ and $y$ belongs to $H$ and $xy$ belongs to $E(h_i)$  for
some $i\in \{1,2\}$. Thus $f$ has the form $HHD_{0}AAD_{0}$.
\end{proof}

\begin{Obser}\label{characterizationm3} Let $f$ be a face of $F$
containing no contributing edges and $f\neq h_{1},f\neq h_{2}$. Then
$f$ is either a pentagon with the boundary $HHD_{0}AD_{0}$ or a hexagon
intersecting both $h_{1}$ and $h_{2}$ and with the boundary
$HHD_{0}HHD_{0}$ or a hexagon with the boundary
$AD_{0}AD_{0}AD_{0}$.
\end{Obser}
\begin{proof} If $f$ is a pentagon,  then by (*), $f$ contains a vertex of  $H$. Thus $f$ is a
neighboring face of $h_{i}$  for some $i\in \{1,2\}$ by the 3-regularity of $F$. Hence $f$ has the form $HHD_{0}AD_{0}$ by
Lemma \ref{intersectingboth}. Now suppose $f$ is a hexagon. Then $V(f)\cap  D^{*}=\emptyset$ as $f$
contains no contributing edges. If $f$ has  a
vertex of  $H$,  then $f$ is a neighboring face of $h_{i}$  for
some $i\in \{1,2\}$,
without contributing edges. Thus $f$ intersects both $h_{1}$ and
$h_{2}$ with the boundary $HHD_{0}HHD_{0}$ by Lemma
\ref{intersectingboth}. Otherwise  all vertices of $f$ belong to $A\cup
D_{0}$ and $f$ has the form $AD_{0}AD_{0}AD_{0}$.
\end{proof}

We continue with some structural lemmas.

\begin{Lemma}\label{characterizationf} Let $F^{*}\in {\cal D}^*$ be a pentagon denoted clockwise by
$v_{1}v_{2}v_{3}v_{4}v_{5}v_{1}$. If $f$ is a neighboring face of $F^{*}$ with
$v_{2}v_{2}'xyv_{1}'v_{1}v_{2}$, then one of the following assertions holds.

(1) $f$ is a hexagon, and (i) $v_{1}',y\in V(h_{i}),v_{2}',x\in
V(h_{3-i})$ for $i\in \{1,2\}$, or (ii) $v_{1}',y\in
V(h_{i}),v_{2}'\in A, x\in D_{0}$  $  (v_{2}', x\in
V(h_{i}), v_{1}'\in A, y\in D_0)$,   or (iii)
$v_{1}',v_{2}'\in A,x,y\in D^{*}\setminus V(F^{*})$, or ($iv$)
$v_{1}',v_{2}',x\in A, y\in D_{0}$ $ (v_{1}',v_{2}',y\in A, x\in D_{0})$.

(2) $f$ is a pentagon ($x=y$),  and ($v$)
$v_{1}',x\in V(h_{i}),v_{2}'\in A$  $(v_{2}',x\in V(h_{i}), v_{1}'\in A$) for $i\in
\{1,2\}$, or ($vi$) $v_{1}',v_{2}'\in A,x\in D_{0}$.
\end{Lemma}

\begin{proof} By Lemma \ref{intersectatoneedge}, $|E(f)\cap
E(F^{*})|=1$. That is, $v_{1}',v_{2}',x,y\notin V(F^{*})$. Let
$f_{1},f_{2},f_{3},f_{4}$ be the neighboring faces of $F^{*}$ along
the edges $v_{2}v_{3},v_{3}v_{4}, v_4v_5,v_{5}v_{1}$, respectively.
Denote by $v_{3}',v_{4}',v_{5}'$ the neighbors of
$v_{3},v_{4},v_{5}$, respectively, not in $F^{*}$.

We first consider the case that  $f$ is a hexagonal face. Since $v_{1}',v_{2}'\in
H\cup A$, we have the following three cases.

\emph{Case 1.} $v_{1}',v_{2}'\in H$. Let $v_1'\in V(h_{i})$ and $v_{2}'\in V(h_{j})$ for  $i, j\in \{1,2\}$. Then
$v_{1}'y\in E(h_{i})$ and $v_{2}'x\in E(h_{j})$.  Lemma \ref{intersectatoneedge} implies that $i\not= j$. So ($i$) holds.

\emph{Case 2.} $v_{1}'\in H,v_{2}'\in A$ or $v_{1}'\in A,v_{2}'\in
H$. By symmetry, we only need to consider the former situation.
Then  $v_{1}'y\in E(h_{i})$ for some $i\in \{1,2\}$ by the 3-regularity of $F$. Further by   Observation
\ref{2connected},  we have that $x\notin H\cup D^{*}$ and  $x\in A\cup D_{0}$. Then
$x\in D_{0}$ and ($ii$) holds. Otherwise,  $x\in A$. Then $f$ would contain
three contributing edges $v_{1}v_{2},v_{2}'x,xy$ belonging to
$E(\mathcal{D}^{*}),E(A,A),E(A,H)$, respectively, contradicting  Lemma \ref{characterizationh}(4).

\emph{Case 3.} $v_{1}',v_{2}'\in A$. Then $x,y\notin H$.  Otherwise   $xy\in E(h_i)$ for some  $i\in\{1,2\}$. That is, $f$ is a neighboring face of $h_{i}$ with two
$E(A,H)$ edges and one $E(\mathcal{D}^{*})$ edge, contradicting  Lemma
\ref{characterizationh}(5). If $x\in A$, then
 $y\notin D^{*}$ by
Observation \ref{2connected}. Also $y\notin A$; otherwise,
$|E(A,A)|\geq3$ and $s(\mathcal{D})\geq1$, contradicting Ineq. (5). So $y\in D_{0}$ and ($iv$) holds. If $x\in
D^{*}\setminus V(F^{*})$, then $y\in D^{*}\setminus V(F^{*})$ by Observation
\ref{2connected}, and  ($iii$) holds. If $x\in D_{0}$, then $y\in
H\cup A$. Obviously, $y\notin H$. So $y\in A$ and ($iv$) holds.

 Now suppose $f$ is a pentagon. Also, $v_{1}',v_{2}'\in H\cup A$.  Further,  either one of $v_{1}',v_{2}'$
belongs to $H$, the other to $A$ or both $v_{1}',v_{2}'$ belong to
$A$. If the former holds, then by symmetry, we may assume
$v_{1}'x\in E(h_{i})$ for $i\in \{1,2\}$ and $v_{2}'\in A$. Hence
($v$) holds.  If the latter holds, then also $x\notin H\cup  D^{*}$.
We claim that  $x\in D_{0}$,  and  ($vi$) holds. Suppose to the
contrary that  $x\in A$. Then $f$ contains three contributing edges
$v_{1}v_{2},v_{2}'x,v_{1}'x$.  Ineqs. (5) and (6) imply that
$E(A,H)=\emptyset$, $E(A,A)=\{v_{2}'x,v_{1}'x\}$, and $s(\cal D)=1$.
That is, there are no other contributing edges except for
$v_{2}'x,v_{1}'x$ and $E(F^{*})$. Moreover, both $f_{1}$ and $f_{4}$
are hexagons by the assumption and do not contain contributing edges
except for $E(F^{*})$. This fact means $v_{3}$ and $v_{5}$ are
incident to $h_{i}$ and $h_{3-i}$, respectively, for $i\in \{1,2\}$
by applying Lemma \ref{lessthan1} and Part (1) of this lemma  to the
faces $f_{1},f_{4}$. Furthermore, $f_{1},f_{2},f_{3},f_{4}$ are
hexagons with the boundaries $D^{*}D^{*}HHD_{0}A$ by Lemma
\ref{characterizationh}(1). Then we have the configuration shown in
Fig.  11. Note at least one neighboring face of $h_{i}$ different
from $f_{1},f_{2}$ is hexagonal in order to prevent the forbidden
subgraph $L$ occurring in $F$. So $p(V(h_{i}))\leq3$. Similarly,
$p(V(h_{3-i}))\leq3$. Hence,
$p(H)+p(E(A,A))+p(D^{*})\leq(3+3)+3+2=11$, contradicting Ineq. (7).
This contradiction verifies the claim.
\end{proof}
\begin{figure}[h]
\begin{center}
\includegraphics[scale=0.8
]{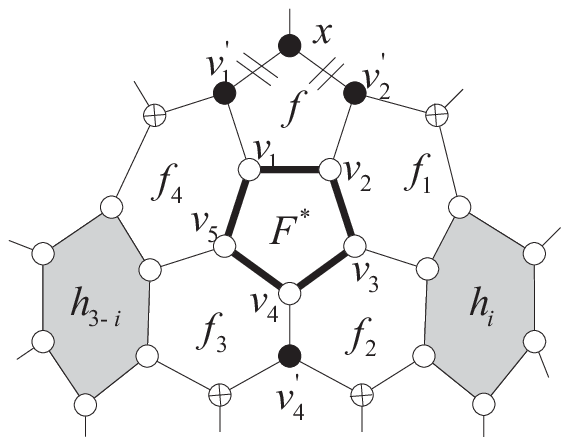}\\{Figure 11. Illustration to the proof of Lemma
\ref{characterizationf}: $f$ is a   pentagon.}
\end{center}
\end{figure}

Some extensions to Lemma \ref{characterizationf} can be obtained in a similar way as follows.

\begin{Lemma}\label{characterizationf1} Let $F^{*}\in \cal D^*$ with a neighboring face $f$. Assume that $f$ a path $a'P(a,b)b'$ on its boundary such that $P(a,b)\subset F^*$ and $a',b'\notin V(F^*)$.

If  $P(a,b)$ is of length one  and  $f$ is a hexagon with $a',b'\in A$, then $\partial(f)$
has the form $D^{*}D^{*}AAD_{0}A$ or $D^{*}D^{*}AD^{*}D^{*}A$.

If $f$ contains  contributing edges belonging
only to $E(\mathcal{D}^{*})$, then $a',b'\in A\cup H$ and one of the following assertions
holds.

($i$) $a'b'\in E(h_{i})$ for some $i\in \{1,2\}$.

($ii$) $a',b'\in A$ and $a',b'$ share a common neighbor belonging to
$D_{0}$ or $f$ is a hexagon with the boundary $D^{*}D^{*}AD^{*}D^{*}A$.

($iii$) one of $a',b'$ belongs to $A$, the other to $H$ and $f$ is a
hexagon with the boundary $D^{*}D^{*}HHD_{0}A$.

\end{Lemma}

Regarding to the non-trivial factor-critical components of $F-(H\cup
A)$, we have the following stronger conclusion.

\begin{Lemma}\label{both} Let $F^{*} \in \cal D^*$  with
$|\nabla(F^*)|\geq7$. Then at least one of $h_{1},h_{2}$ is
incident to $F^{*}$. In particular, if $|\nabla(F^*)|=9$, then
both $h_{1}$ and $h_{2}$ are incident to $F^{*}$.
\end{Lemma}

\begin{proof} Suppose that $h_{i}$ is not incident
to $F^{*}$ for some $i\in \{1,2\}$. By the first part of Lemma \ref{characterizationf1} we have that  the neighboring faces of $h_{i}$ contain no
$E(F^{*})$ edges. On the other hand,  there must exist  contributing edges contained in  the neighboring faces of each of $h_{1}$ and $h_{2}$ by Corollary
\ref{atleastoneedge}.  So $h_i$ has a neighboring face with a contributing edge not in $E(F^*)$.

If $|\nabla(F^*)|=9$,  Ineqs. (5) and (6) imply  that  $F$ has only contributing $E(F^*)$ edges.  Hence both $h_{1}$ and $h_{2}$ have a neighboring face with $E(F^*)$ edge.  This contradiction shows that
both $h_{1}$ and $h_{2}$ are incident to $F^{*}$.

In the following suppose $|\nabla(F^*)|=7$.  Suppose to the contrary that none of $h_{1}$ and $h_2$ are  incident
to $F^{*}$.  Then  $s(\mathcal{D})\geq s(F^{*})=2$.  Ineqs. (5) and (6)
imply exactly one additional contributing edge in
$E(h_{1},h_{2})$,  $E(A,H)$,   $E(A,A)$),  or one
 $F_{1}^{*}\in \cal D^*$ with
$|\nabla(F_{1}^{*})|=5$. For one
$E(h_{1},h_{2})$ edge existence, we can know $h_{1}$ and
$h_{2}$ are incident and all of the neighboring faces of $h_{1}$ and
$h_{2}$ are pentagons by Lemmas \ref{intersectingboth} and
\ref{characterizationh}(2). Immediately the subgraph $L$ occurs in
$F$, contradicting the assumption. For one $E(A,H)$ edge existence,
the neighboring faces of $h_{i}$ for some $i\in \{1,2\}$ don't
contain contributing edges, contradicting  Corollary
\ref{atleastoneedge}. For one $E(A,A)$ edge existence, the $E(A,A)$
edge must be contained in one neighboring face of each  of  $h_{1}$ and
$h_{2}$ , which is   a hexagon by Lemma \ref{characterizationh}(1). Then
the remaining five neighboring faces of $h_{i}$ for $i\in \{1,2\}$
are either pentagons or hexagons intersecting both $h_{1}$ and
$h_{2}$ by Lemma \ref{intersectingboth}. So one of the remaining
five neighboring faces of $h_{i}$ must be a hexagon intersecting both $h_{1}$ and
$h_{2}$ in order to avoid the occurrence of the forbidden subgraph
$L$, which is impossible by Lemma \ref{equalemptyset}. For
$F_{1}^{*}$ existence, $F_1^*$ is a pentagon by Prop. \ref{5edgecut}, and both $h_{1}$ and $h_{2}$ are incident to
$F_{1}^{*}$  and four neighboring faces of $F_{1}^{*}$ are hexagons
with the boundaries $D^{*}D^{*}HHD_{0}A$ by Lemma
\ref{characterizationh}(1). By Lemma \ref{intersectingboth},  $h_{1}$ and $h_{2}$ have a common neighboring face $HHD_0HHD_0$ in order to prevent
the occurrence of $L$, again contradicting Lemma
\ref{equalemptyset}.
\end{proof}

\begin{cor}\label{both1} Let $F^{*}\in \cal D^*$ with $|\nabla(F^*)|=7$. Assume that $F^{*}$ is not
incident to $h_{i}$ for some $i\in\{1,2\}$. Then both $h_{i}$ and
$h_{3-i}$ have a common
neighboring face with the boundary $HHD_{0}HHD_{0}$ and one of the
remaining five neighboring faces of $h_{i}$ contains either an
$E(A,V(h_{i}))$ edge such that $p(V(h_{i}))\leq3$ or an $E(A,A)$
edge with the boundary $HHD_{0}AAD_{0}$ or an $E(F_{1}^{*})$ edge
such that $F_{1}^{*}\in \cal D^*$ is a pentagon and
$p(V(h_{i}))\leq3$.
\end{cor}

\begin{proof}From  Lemma  \ref{both} and its proof we can know if $h_{i}$ for $i\in \{1,2\}$ is not
incident to $F^{*}$  with $|\nabla(F^*)|=7$, then $h_{3-i}$ is
incident to $F^*$ and $h_1$ and $h_2$ are not incident. There is
exactly one  $E(A,V(h_{i}))$ or $E(A,A)$ or $E(F_{1}^{*})$ edge,
contained in the neighboring faces of $h_{i}$. Note that an
$E(A,V(h_{i}))$ (or $E(F_{1}^{*})$) edge gives rise to two adjacent
hexagonal neighboring faces of $h_{i}$ by Lemma
\ref{characterizationh}(1).  For $E(F_{1}^{*})$) edge case, $h_i$ is
connected by exactly one edge by Lemma \ref{lessthan1}. Hence the
other four  neighboring faces of $h_i$ have no contributing edges.
Thus in order to prevent the occurring of the forbidden subgraph $L$
one of the neighboring faces of $h_{i}$ must be a hexagon
$HHD_{0}HHD_{0}$ intersecting both $h_{1}$ and $h_{2}$ by Lemma
\ref{intersectingboth}.  For $E(A,A)$ edge case, the proof
is similar.
\end{proof}

By applying the above preliminary results we can obtain the following three critical lemmas. But their proofs will be presented  in Sections 4, 5 and 6, respectively.

\begin{Lemma}\label{no9edgecut} There is no  $F^{*}\in\cal D^*$  with $|\nabla(F^*)|=9$.
\end{Lemma}

\begin{Lemma}\label{no7edgecut} There is no $F^{*}\in\cal D^*$ with $|\nabla(F^*)|=7$.
\end{Lemma}

\begin{Lemma}\label{no5edgecut} There is no $F^{*}\in\cal D^*$ with $|\nabla(F^*)|=5$.
\end{Lemma}

We will  complete the proof by producing contradictions in all cases, where graph $F_{42}$ in Fig. 2 is excluded. In fact the other ten graphs are excluded in proving Lemmas \ref{no9edgecut} to \ref{no5edgecut}.

Before we have already shown that  $3\leq|\nabla(F^*)|\leq 9$ for
each $F^*\in\cal D$. Lemmas  \ref{no9edgecut} to \ref{no5edgecut}
imply that  every component $F^{*}$ in $\mathcal{D}$ sends out
exactly three edges. If $F^{*}$ is non-trivial, then $\nabla(F^*)$
forms a cyclic 3-edge-cut by Observation \ref{2connected},
contradicting that $c\lambda(F)=5$. So the components in
$\mathcal{D}$ are singletons and we have the following claim.

{\bf Claim 1}. $D^*=\emptyset$.

Then  $V(F)=H\cup A\cup D_{0}$ and
 $F-H-E(A,A)$ is bipartite. So each  pentagon of $F$
must contain a vertex in $H$ or an edge in $E(A,A)$. By (7) and (8) we have
$p(H)\leq8$, which means $p(E(A,A))\geq4$. On the other hand,
$p(E(A,A))\leq2|E(A,A)|\leq6$ as $|E(A,A)|\leq3$ by Ineq.
(5). So $4\leq p(E(A,A))\leq6$.

{\bf Claim 2.} $p(E(A,A))=4$.

\begin{proof}If $p(E(A,A))\geq5$, then $|E(A,A)|=3$, say
$e_{1},e_{2},e_{3}\in E(A,A)$. Moreover, at least two of
$e_{1},e_{2},e_{3}$, say $e_{1},e_{2}$, belong to two pentagons. In
other words, $e_{1},e_{2}$ cannot be contained in the neighboring
faces of $h_{1}$ or $h_{2}$ by Lemma \ref{characterizationh}(1) and
(3). Hence the neighboring faces of $h_{i}$ for some $i\in\{1,2\}$
do not include contributing  edges as $p(\{e_{3}\})\geq1$,
contradicting Corollary \ref{atleastoneedge}.
\end{proof}

{\bf Claim 3.} $|E(A,A)|=3$.

\begin{proof}By Claim 2 we have $|E(A,A)|\geq2$. If
$|E(A,A)|=2$, then each of these two $E(A,A)$ edges belongs to two
pentagons and none of these two $E(A,A)$ edges is contained in the
neighboring faces of $h_{j}$ for $j\in \{1,2\}$ by Lemma
\ref{characterizationh}(1) and (3). So there is exactly one other
contributing edge belonging to $E(A,H)$ or $E(h_{1},h_{2})$ in
$F$ by Corollary \ref{atleastoneedge} and Ineqs. (5) and (6). If this
additional contributing edge belongs to $E(A,H)$, then it is in the neighboring faces of exactly one of $h_1$ and $h_2$, contradicting  Corollary \ref{atleastoneedge}. Otherwise,  $h_{1}$ and $h_{2}$ are incident, we
can obtain a subgraph $L$ in $F$ applying
Lemmas \ref{characterizationh}(2) and \ref{intersectingboth} to the neighboring faces of
 $h_{1}$ and $h_{2}$ (In fact, we can obtain stronger result by Lemma \ref{aforest}: the neighboring faces of $h_i$ are pentagons), contradicting the assumption.
\end{proof}
Put  $\{e_{1},e_{2},e_{3}\}= E(A,A)$. Then we have the following
result.

\textbf{Claim 4.}  Each $e_{i}$, $i\in \{1,2,3\}$, cannot be the intersection of two
hexagonal faces of $F$ for $i\in \{1,2,3\}$.

\begin{proof}  Suppose to the contrary that   $e_{1}$ belongs to two
hexagons $f_{1}$ and $f_{2}$. Then none of $e_{2}$ and $e_{3}$ can
be contained in $f_{1}$ or $f_{2}$; Otherwise,
$p(E(A,A))\leq3$,
contradicting that $p(E(A,A))=4$. In other words, both $f_{1}$
and $f_{2}$ include only one $E(A,A)$ edge. Then they are the
neighboring faces of $h_{1}$ and $h_{2}$ by
Observation \ref{characterizationm2}. On the other hand, as
$p(E(A,A))=4$ and $p(\{e_{1}\})=0,p(\{e_{2}\})=p(\{e_{3}\})=2$,
which means the neighboring faces of $h_{1}$ and $h_{2}$ but
different from $f_{1}$ and $f_{2}$ contain no contributing edges by
Lemma \ref{characterizationh}(1) and (3). So in order to prevent the
subgraph $L$ occurring in $F$, by Lemma  \ref{intersectingboth} we have that one of the neighboring faces of
$h_{1}$ must be a hexagon intersecting  $h_{2}$,
which is impossible by Lemma \ref{equalemptyset}.
\end{proof}

 \textbf{Claim 5.}       Each  pentagonal face $f$ of $F$ contains at
               most one of $e_{1},e_{2},e_{3}$.

\begin{proof} Let $uvwxyu$ be the pentagon  $f$. If exactly two of $e_{1},e_{2},e_{3}$ are
contained in $f$, say $uv,vw\in E(A,A)$, then $x,y\in D_{0}$,
contradicting  that $D_0$ is independent in $F$. If all of $e_{1},e_{2},e_{3}$ are contained in $f$, say $uv,vw, wx\in E(A,A)$ and $y\in D_0$, then
the three neighboring faces of $f$ (containing $e_{1},e_{2},e_{3}$,
respectively) are pentagonal as $p(E(A,A))=4$.  Then $F$ includes $R$, a contradiction.
\end{proof}

From  Claims 2 and 3, we can know there is precisely one $E(A,A)$ edge
(say $e_{3}$) belonging to two pentagons. From Claim 4 each of $e_{1},e_{2}$
is the intersection of one pentagon and one hexagon. By Lemma
\ref{characterizationh}(1), $e_{3}$ cannot be contained in the
neighboring faces of $h_{1}$ and $h_{2}$. Since $p(V(h_{j}))=4$, at
least two neighboring faces of $h_{j}$ ($j\in \{1,2\}$) are
hexagons. Thus in order to form the hexagonal neighboring faces of
$h_{j}$ and $h_{3-j}$, $e_{1},e_{2}$ must be contained in the
neighboring faces of $h_{j}$ and $h_{3-j}$ (respectively) with the
boundaries $HHD_{0}AAD_{0}$ and there exists a neighboring face of
$h_{j}$ intersecting both $h_{j}$ and $h_{3-j}$ with the boundary
$HHD_{0}HHD_{0}$. More precisely, only two neighboring faces of
$h_{j}$ either containing the edge $e_{i}$ or intersecting both
$h_{j}$ and $h_{3-j}$ for $i,j\in \{1,2\}$ are hexagonal and the
remaining four are pentagonal with the boundaries $HHD_{0}AD_{0}$.
Hence the two hexagonal neighboring faces of $h_{j}$ cannot be
adjacent in order to prevent the subgraph $L$ occurring in $F$ for
$j\in \{1,2\}$. Denote by $f_{i}$($g_{i}$) ($1\leq i\leq6$) the
neighboring faces of $h_{j}$($h_{3-j}$) in clockwise
(anti-clockwise). Without loss of generality, suppose $f_{2}=g_{2}$
(see Fig.  12). By Lemma \ref{equalemptyset}, the $f_i$ and $g_j$
are disjoint,  $4\leq i,j\leq 6$, and
 $e_{1}$ belongs to some  $f_i$ and $e_2$  to some $g_j$ by the
above analysis.  So there are four cases for
 distributions of $e_{1},e_{2}$ on
$f_{4},f_{5},f_{6}$ and $g_{4},g_{5},g_{6}$ by symmetry (see Fig.
12(a),(b),(c),(d)). Let $f_{7}$ and $f_{8}$ be two faces of $F$
adjacent with $f_{3},f_{4},g_{3},g_{4}$ and
$f_{1},f_{6},g_{1},g_{6}$, respectively. If the case (a) or (b) or
(c) holds, then a subgraph $L$ consisting of
$f_{3},f_{4},g_{3},g_{4}$ occurs, a contradiction. If the case (d)
holds, then we have the configuration as shown in Fig.  12(d) by the
above discussions, where  $f_{7},f_{8}$ must be hexagons from Claim 5 as there is
no $R$ in $F$ and $v_{1},v_{2}\in D_{0}$. Let $G=h_{j}\cup
h_{3-j}\bigcup_{i=1}^{6} f_{i}\bigcup_{i=1}^{6} g_{i}\cup f_{7}\cup
f_{8}$ be a fragment of $F$. Denote by $f_{9},\cdots,f_{14}$ the
neighboring faces of $G$ as shown in Fig.  12(d). As
$f_{1},g_{1},g_{6}$ are pentagonal, $f_{9}$ must be a hexagon by the
assumption. So does $f_{12}$ by symmetry. Hence $v_{3},v_{4}\in A$
and $v_{5},v_{6}\in D_{0}$ by Observation \ref{characterizationm3}
(see Fig.  12(e)). Moreover, $f_{11}$ and $f_{14}$ are pentagons as
$p(\{e_{1}\})=p(\{e_{2}\})=1$. Thus $v_{7},v_{8}\in D_{0}$. Since
$D_0$ is independent in $F$,  $v_{6}$ and $v_{8}$ ($v_{5}$ and
$v_{7}$) must share a common neighbor, say $v_{9}$ ($v_{10}$),
belonging to $A$ (see Fig.  12(e)). Finally $v_{9}$ and $v_{10}$
must be adjacent by Lemma \ref{aforest} and we have the fullerene
graph $F_{42}$, which is excluded in the assumption. Until now
Theorem \ref{lr} is completed.

\begin{figure}[h]
\begin{center}
\includegraphics[scale=0.85
]{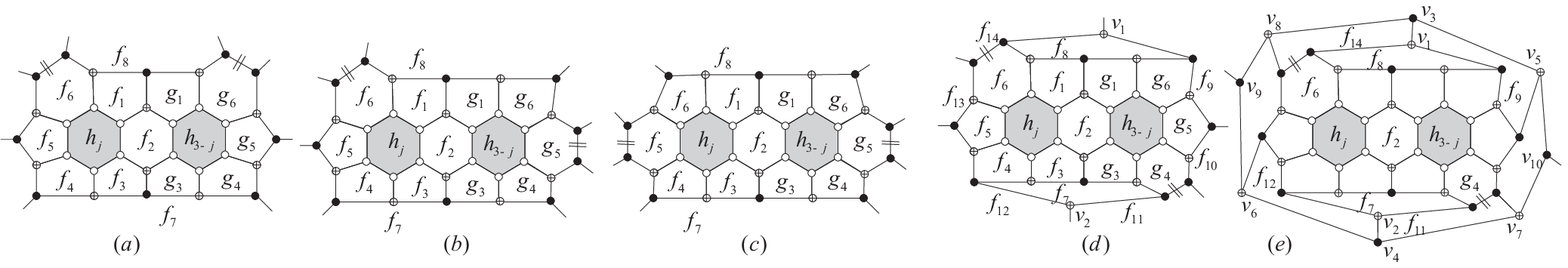}\\{Figure 12.  The four cases for distributions of
$e_{1},e_{2}$  in $f_{4},f_{5},f_{6}$ and $g_{4},g_{5},g_{6}$.}
\end{center}
\end{figure}

%%%%%%%%%%%%%%%%%%%%%%%%%%%%%%%%%%%%%%%%%%%%%%%%%%%%%%%%%%%%%%%%%%%%%%%%%%%%%%%%%%
%%%%%%%%%%%%%%%%%%%%%%%%%%%%%%%%%%%%%%%%%%%%%%%%%%%%%%%%%%%%%%%%%%%%%%%%%%%%%%%%%%

\section{Proof of Lemma \ref{no9edgecut}}

Suppose to the contrary that $F^{*}\in \cal D$ with $|\nabla(F^*)|=9$ exists. Then by
Ineqs. (5) and (6), $\cal D^*$ consists only of $F^{*}$, $E(A,A)=0$ (so $A$ is an independent set), $E(A,H)=0$ and $|\nabla(H)|=12$ (in particular, $h_{1}$ and
$h_{2}$ are not incident). Thus both $h_{1}$ and $h_{2}$ are
incident to $F^{*}$ by Lemma \ref{both}. For  $j\in \{1,2\}$, let
$k=|\nabla(h_{j})\cap \nabla(F^*)|$. Obviously, $1\leq
k\leq6$. Denote by $v_{1}v_{2}v_{3}v_{4}v_{5}v_{6}v_{1}$ the
boundary of $h_{j}$ along the clockwise direction and
$f_{1},f_{2},\cdots,f_{6}$ the six neighboring faces of $h_{j}$
containing the edges $v_{1}v_{2},v_{2}v_{3},\cdots,v_{6}v_{1}$,
respectively. Let $G=h_{j}\cup \bigcup_{i=1}^{6} f_{i}$. Before our
main argument, we give a definition for clusters.

A {\em cluster} at $h_{j}$ ($j=1,2$) is a sequence $Q=(e_{1},\cdots,
e_{r})$ such that

(1)  $e_{i}$ and $e_{i+1}$ belong to the same face  for each $1\leq i\leq
r-1$,

(2) $r\geq 3$ and for $2\leq i\leq r-1$, $e_{i}\in \nabla(F^*)\cap
\nabla(h_{j})$, while $e_{1},e_{r}\in \nabla(F^*)\backslash
\nabla(h_{j})$.

The size $|Q|$ of $Q$ is the number of its edges. Observe that
$e_{1}\neq e_{r}$ by Lemma \ref{intersectatoneedge} and $|Q|\geq3$.

For two edges $e_{1},e_{2}$ of
$F$, we call $e_{1}$ is {\em opposite} $e_{2}$ if they both belong
to the same hexagonal face of $F$ and no edge of this boundary
is incident with both $e_{1}$ and $e_{2}$.

About the clusters we have the following properties.

\textbf{Claim 1.} Let $j\in \{1,2\}$. Assume
that the neighboring faces of $h_{j}$ contain no other contributing
edges except for $E(\mathcal{D}^{*})$, then clusters at $h_{j}$ are
pairwise disjoint.
\begin{proof}  The proof of (1) is the same as Lemma 11 in
\cite{kaiser2010}. We omit it here. \end{proof}

\textbf{Claim 2.} There is at most one edge that is contained in a
cluster $Q_{1}$ at $h_{j}$ and a cluster $Q_{2}$ at $h_{3-j}$ such
that $|Q_{1}|=3,|Q_{2}|\geq3$.
\begin{proof}
 Assume two edges $e_{1},e_{2}$ are contained in both $Q_{1}$ and
$Q_{2}$. The definition of a cluster and Lemma
\ref{characterizationf1} guarantee $e_{1}$($e_{2}$) cannot be
contained in $\nabla(h_{1})$ or $\nabla(h_{2})$ and both $e_{1}$ and
$e_{2}$ are opposite an edge of $h_{j}$ and opposite an edge of
$h_{3-j}$. If $|Q_{1}|=3$ and $|Q_{2}|=3$, then by the above
analysis and planarity of $F$ we can obtain the configuration shown
in Fig. 13(a). Immediately a quadrangular face occurs, contradicting
the definition of $F$. If $|Q_{1}|=3$ and $|Q_{2}|\geq4$, also a
configuration depicted in Fig.  13(b) can be gained. However, in
this case, a degree-saturated path of length more than six is
obtained (see Fig.  13(b) the path along
$v_{1},v_{2},\cdots,v_{7}$), contradicting that every face of $F$
has a size of at most six.
\end{proof}

\begin{figure}[h]
\begin{center}
\includegraphics[scale=0.8
]{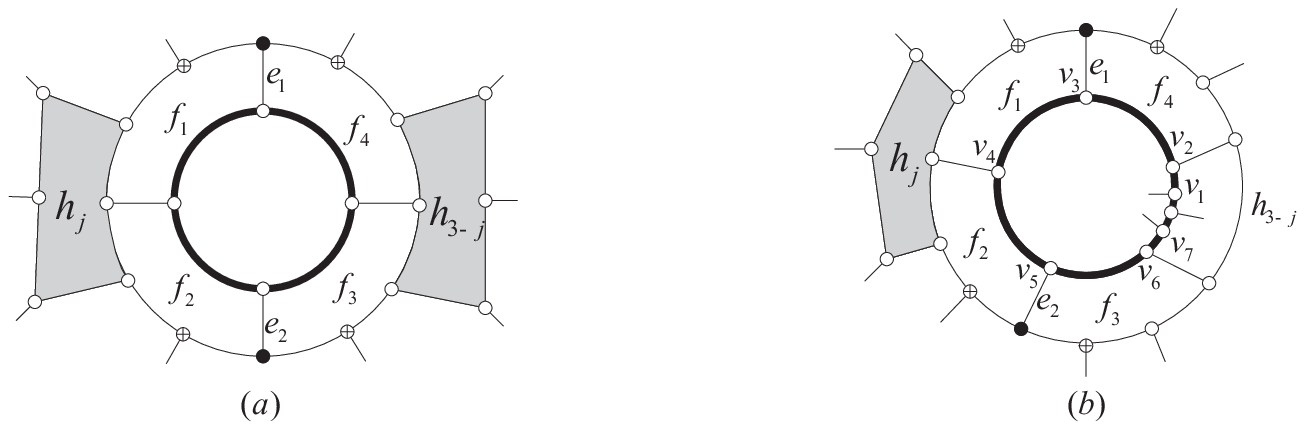}\\{Figure 13.  Clusters $Q_{1},Q_{2}$ at $h_{j}$ and
$h_{3-j}$: (a) $|Q_{1}|=3,|Q_{2}|=3$,  (b)
$|Q_{1}|=3,|Q_{2}|\geq4$.}
\end{center}
\end{figure}

Next we distinguish the following cases to complete the proof of
Lemma \ref{no9edgecut}.

\emph{Case 1.} $k=1$. Then there is exactly one cluster $Q$ at
$h_{j}$ with $|Q|=3$. Without loss of generality, suppose $v_{4}$ is
incident to $F^{*}$(see Fig.  14(a)). Then all of
$v_{1},v_{2},v_{3},v_{5},v_{6}$ are incident to $D_{0}$ and
$f_{3},f_{4}$ are hexagons with the boundaries $D^{*}D^{*}HHD_{0}A$
by Lemma \ref{characterizationf1}. In order to prevent the forbidden
subgraph $L$ occurring in $F$, at least one of
$f_{1},f_{2},f_{5},f_{6}$ is hexagonal. Lemmas
\ref{intersectingboth} and \ref{atmostoneneighboring} imply exactly
one of $f_{1},f_{2},f_{5},f_{6}$ must intersect both $h_{j}$ and
$h_{3-j}$ with the boundary $HHD_{0}HHD_{0}$ to form this hexagonal
face.

If $f_{2}$ intersects both $h_{j}$ and $h_{3-j}$, then $f_{1},f_{5}$
and $f_{6}$ are pentagons and $a_{i}$($1\leq i\leq5$) belongs to $A$
by Lemmas \ref{intersectingboth} and \ref{characterizationf1}, where
$a_{i}$($1\leq i\leq5$) is shown in Fig.  14(a). Let
$w_{1}w_{2}w_{3}w_{4}w_{5}w_{6}w_{1}$ be the boundary of $h_{3-j}$
along the anti-clockwise direction such that $\partial(f_{2})\cap
\partial(h_{3-j})=\{w_{2}w_{3}\}$. Denote by
$f_{7},f_{8},\cdots,f_{13}$ the seven neighboring faces of $G$ as
depicted in Fig.  14(a). If $a_{1}$ is incident to $F^{*}$, then all
of $a_{2},a_{3},a_{4}$ and $w_{1}$ are incident to $F^{*}$ to form
the neighboring faces $f_{7},f_{8},f_{9},f_{10}$ of $F^{*}$ by Lemma
\ref{characterizationf1}(see Fig.  14(b)). Now we obtain eight edges
belonging to $\nabla(F^*)$. Thus $w_{4}$ cannot be incident to
$F^{*}$, otherwise, $a_{5}$ is again incident to $F^{*}$ to form the
neighboring face $f_{12}$ of $F^{*}$ by Lemma
\ref{characterizationf1} and $|\nabla(F^*)|>9$, contradicting that
$|\nabla(F^*)|=9$. That is, $w_{4}$ is incident to $D_{0}$ and
$a_{5}$ and $w_{4}$ share a common neighbor, say $d_{1}$ (see Fig.
14(b)). It's easy to see $d_{1}$ cannot be again incident to
$h_{3-j}$, which means $d_{1}$ is again incident to a vertex
belonging to $A$, say $a_{6}$. Furthermore, $a_{6}$ must be incident
to $F^{*}$ to form the neighboring face $f_{13}$ of $F^{*}$ by Lemma
\ref{characterizationf1}. Then we obtain all the nine edges in
$\nabla(F^*)$ (see Fig.  14(b)) and three 2-degree vertices
$w_{6},w_{5}$ and $a_{6}$ that should have a common neighbor by
Lemma \ref{aforest}. Immediately a triangular face occurs, which is
impossible. So $a_{1}$ is incident to $D_{0}$ third times and shares
a neighbor with $w_{1}$. Moreover, $a_{2},a_{3},a_{4}$ are also
incident to $D_{0}$ third times as there are no edges between
$D_{0}$ and $F^{*}$ and all of $f_{8},f_{9},f_{10}$ are hexagons
with the boundaries $AD_{0}AD_{0}AD_{0}$ by
Observation\ref{characterizationm3}. Let $a_{6},a_{7},a_{8},d_{1}$
be the vertices shown in Fig.  14(c). If $a_{6}$ is incident to
$D_{0}$ third times, then the neighboring face of $h_{3-j}$ (say
$f_{14}$) including the edge $w_{1}w_{6}$ is a pentagon by
Lemma\ref{intersectingboth}. Thus $f_{14}\cup f_{7}\cup f_{1}\cup
f_{6}$ forms a subgraph $L$ in $F$ (see Fig.  14(c)), contradicting
the assumption. So $a_{6}$ is incident to $F^{*}$. Similarly,
$w_{6},a_{7},a_{8}$ are also incident to $F^{*}$. By the planarity
of $F$, $d_{1}$ is not incident to $h_{3-j}$. That is, $d_{1}$ is
incident to $A$ third times. Let $a_{9}$ be the third neighbor of
$d_{1}$. Then $a_{9}$ is incident to $F^{*}$ with two edges to form
two neighboring faces of $F^{*}$ also by Lemma
\ref{characterizationf1}(see Fig.  14(d)). We once again obtain nine
edges in $\nabla(F^*)$ and three 2-degree vertices
$w_{5},w_{4},a_{5}$ that must share a common neighbor. Also a
triangular face occurs, a contradiction.

\begin{figure}[h]
\begin{center}
\includegraphics[scale=0.7
]{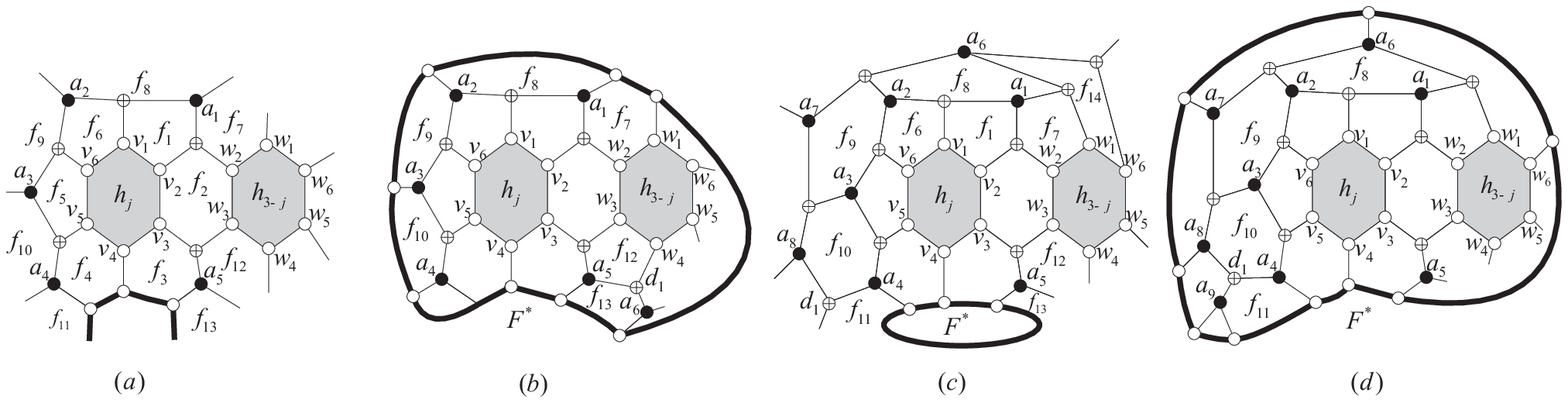}\\{Figure 14.  Illustration for Case 1 in the proof of
Lemma \ref{no9edgecut}.}
\end{center}
\end{figure}

So $f_{2}$ cannot intersect both $h_{j}$ and $h_{3-j}$. Thus it is a
pentagon by Lemma \ref{intersectingboth}. By symmetry, $f_{5}$ is
also a pentagon. If $f_{1}$ intersects both $h_{j}$ and $h_{3-j}$,
then similarly as the case above, we always obtain nine edges in
$\nabla(F^*)$ and three 2-degree vertices that must share a
common neighbor and finally occurs a triangular face, which is
impossible.

Summarizing the above analysis, both $h_{1}$ and $h_{2}$ must be
incident to $F^{*}$ with at least two edges, that is, $k\geq2$.

\emph{Case 2.} $k=2$. Then there are at most two clusters at
$h_{j}$. If exactly one cluster $Q_{1}$ exists at $h_{j}$. The
definition of a cluster implies $|Q_{1}|=4$. Without loss of generality, we may assume $v_{3},v_{4}$ are incident to $F^{*}$.
Analogously as Case 1, $f_{2},f_{4}$ are hexagons with the
boundaries $HHD^{*}D^{*}AD_{0}$ and $f_{1},f_{5},f_{6}$ are either
pentagons with the boundaries $HHD_{0}AD_{0}$ or hexagons
intersecting both $h_{j}$ and $h_{3-j}$ with the boundary
$HHD_{0}HHD_{0}$. For $f_{i}$ ($i=1,5,6$) intersecting both $h_{j}$
and $h_{3-j}$, we always obtain a triangular face in $F$, which is
impossible. For $f_{1},f_{5},f_{6}$ being pentagonal, we have the
configuration shown in Fig.  15(a). Let $a_{i}$ and $f_{j}$ be shown
in Fig.  15(a) for $i\in\{1,\cdots,5\},j\in\{7,\cdots,12\}$. Then
$a_{i}\in A$. If one of $a_{1},\cdots,a_{5}$ is incident to $F^{*}$,
then the remaining four are also incident to $F^{*}$ and we have
nine edges in $\nabla(F^*)$, but now $h_{1}$ and $h_{2}$ are missed,
which is impossible. So all of $a_{1},\cdots,a_{5}$ are incident to
$D_{0}$ third times and again $f_{8},f_{9},f_{10},f_{11}$ are
hexagons with the boundaries $AD_{0}AD_{0}AD_{0}$. Let
$a_{6},\cdots,a_{9},d_{1},d_{2}$ and $f_{13},\cdots,f_{17}$ be the
vertices and faces of $F$ as depicted in Fig.  15(b). If $d_{1}$ is
incident to $h_{3-j}$, then $f_{12}$ is a hexagon with the boundary
$D^{*}D^{*}HHD_{0}A$ by Lemma \ref{characterizationf1} and all of
$a_{6},\cdots,a_{9}$ are incident to $D_{0}$ third times, otherwise,
$|\nabla(F^*)|>9$ (impossible). That is, $f_{13}$ is a pentagon with
the boundary $HHD_{0}AD_{0}$ and $f_{14},f_{15},f_{16}$ are hexagons
with the boundaries $AD_{0}AD_{0}AD_{0}$ (see Fig.  15(c)).
Moreover, $d_{2}$ can not be incident to $h_{3-j}$ by the planarity
of $F$. That is, $d_{2}$ is incident to $A$ third times, say
$a_{10}\in A$ (see Fig. 15(c)). Again $a_{10}$ is incident to
$F^{*}$ to form the neighboring face $f_{7}$ of $F^{*}$ and $f_{17}$
is also a hexagon with the boundary $AD_{0}AD_{0}AD_{0}$. Now we
have six edges in $\nabla(F^*)$ (see Fig.  15(c)). Let
$a_{11},\cdots,a_{14},d_{3}$ be shown in Fig.  15(c). Apply the same
analysis to $a_{11},\cdots,a_{14},d_{3}$ as
$a_{6},\cdots,a_{9},d_{2}$ and repeat this procedure, finally we can
gain nine edges in $\nabla(F^*)$ but only one of them belongs to
$\nabla(h_{3-j})$, contradicting that $|\nabla(h_{3-j})\cap
\nabla(F^*)|\geq2$. This contradiction means $d_{1}$ is incident to
$A$ third times, say $a_{10}$. Similarly for $d_{2}$ (see Fig.
15(d)). Moreover, $d_{2}$ can not be incident to $a_{10}$ by the
planarity of $F$ and Lemma \ref{aforest}. Now we have a similar
situation as the vertices $a_{1},\cdots,a_{5}$ and we can repeat the
above procedure to $a_{6},\cdots,a_{11}$ until we gain nine edges in
$\nabla(F^*)$, but $h_{3-j}$ isn't incident to $F^{*}$,
contradicting Corollary \ref{atleastoneedge}.

\begin{figure}[h]
\begin{center}
\includegraphics[scale=0.7
]{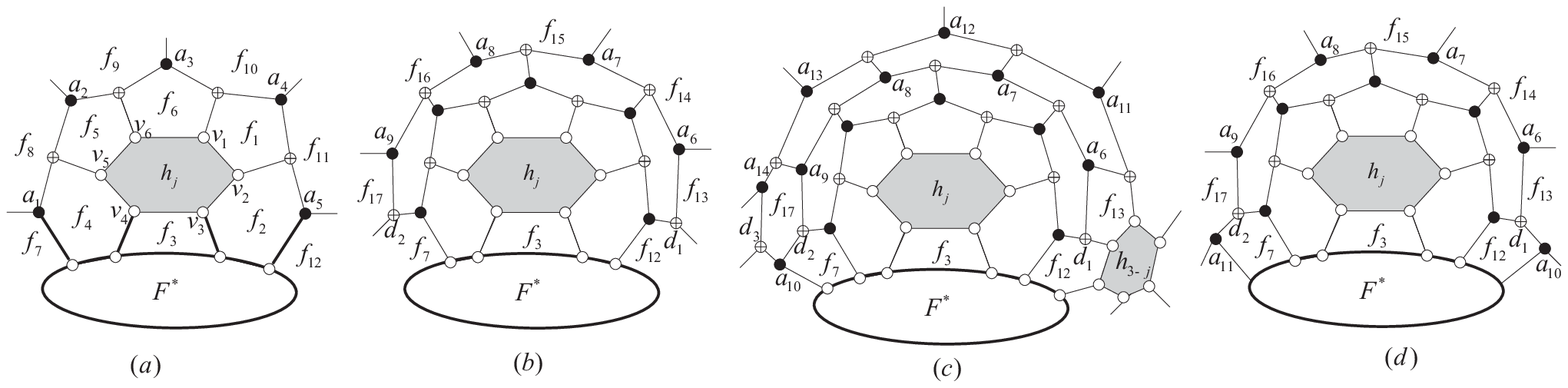}\\{Figure 15. Illustration for Case 2 in the proof of
Lemma \ref{no9edgecut}.}
\end{center}
\end{figure}

If there are two clusters $Q_{1},Q_{2}$ at $h_{j}$. As
$|\nabla(h_{j})\cap \nabla(F^*)|=2$, the definition of a cluster and
Claim 1 imply $|Q_{1}|=3,|Q_{2}|=3$ and we have obtained six edges
in $\nabla(F^*)$. Thus there are at most three edges of
$\nabla(F^*)$ left for the clusters at $h_{3-j}$. Note that there
are at most two edges belonging to the intersection of two clusters
at $h_{j}$ and $h_{3-j}$, respectively, and these two edges cannot
be in $\nabla(h_{j})$ or $\nabla(h_{3-j})$. So at most two clusters
exist at $h_{3-j}$. Moreover, the former situation shows that there
can not be one cluster at $h_{3-j}$ of size 4. Summarizing the above
analysis, we can know either precisely one cluster, say $Q_{3}$,
exists at $h_{3-j}$ satisfying $|Q_{3}|=5$ and $|Q_{3}\cap Q_{i}|=2$
for some $i\in \{1,2\}$ or two clusters, say $Q_{3},Q_{4}$, exist at
$h_{3-j}$ such that $|Q_{3}|=|Q_{4}|=3$ and $|Q_{i}\cap Q_{3}|=2$ or
$|Q_{3-i}\cap Q_{4}|=2$ for some $i\in \{1,2\}$ or
$|Q_{3}|=3,|Q_{4}|=4$ and $|Q_{i}\cap Q_{3}|=2, |Q_{3-i}\cap
Q_{4}|=2$ for $i\in \{1,2\}$. However, no matter which case happens
will contradict Claim 2. So $k\neq2$. In other words,
$|\nabla(h_{i})\cap \nabla(F^*)|\geq3$ for all $i\in \{1,2\}$.

\emph{Case 3.} $k\geq 3$. Let $k=3$. Then also at most two clusters
at $h_{j}$ are obtained by Claim 1 and the fact that
$|\nabla(h_{3-j})\cap \nabla(F^*)|\geq3$. As before, for exactly one
cluster at $h_{j}$ existence, we can always obtain a triangular face
in $F$ by checking for $f_{i}$ ($1\leq i\leq6$) intersecting both
$h_{j}$ and $h_{3-j}$ or not (Note in this case
$|\nabla(h_{3-j})\cap \nabla(F^*)|\geq3$). For two clusters
$Q_{1},Q_{2}$ at $h_{j}$ existence, we have $|Q_{k}|=3$ and
$|Q_{3-k}|=4$ for $k=1,2$ and $Q_{k}$ and $Q_{3-k}$ are disjoint by
Claim 1. That is, there are at most two edges of $\nabla(F^*)$ left
for the clusters at $h_{3-j}$, contradicting the fact that
$|\nabla(h_{3-j})\cap \nabla(F^*)|\geq3$. So $k\geq 4$ . However, in
this case we once again obtain that $|\nabla(F^*)|>9$, also a
contradiction.

Summarizing the above discussion, such a $F^{*}$ cannot exist in
$F$.

%%%%%%%%%%%%%%%%%%%%%%%%%%%%%%%%%%%%%%%%%%%%%%%%%%%%%%%%%%%%%%%%%%%%%%%%%%%%%
%%%%%%%%%%%%%%%%%%%%%%%%%%%%%%%%%%%%%%%%%%%%%%%%%%%%%%%%%%%%%%%%%%%%%%%%%%%%%
%%%%%%%%%%%%%%%%%%%%%%%%%%%%%%%%%%%%%%%%%%%%%%%%%%%%%%%%%%%%%%%%%%%%%%%%%%%%%
\section{Proof of Lemma \ref{no7edgecut}}

Suppose to the contrary that $F^{*}\in \mathcal{D}$ with $|\nabla(F^{*})|=7$
exists. Then $\nabla(F^*)$ forms a cyclic 7-edge-cut. If
$\nabla(F^*)$ is a degenerate cyclic 7-edge-cut, then $F^{*}$ or
$\overline{F^{*}}$ contains less than six pentagons, which means
$F^{*}$ or $\overline{F^{*}}$ is isomorphic to one component of
$D01,\cdots,D57$ as shown in Fig.  5. However, $F^{*}$ or
$\overline{F^{*}}$ is only possibly isomorphic to the components of
$D01,\cdots,D09,D11$ as $F$ can not possess $L$ or $R$ as subgraphs.
Moreover, since $F^{*}$ is 2-connected and $\overline{F^{*}}$
contains two disjoint hexagons $h_{1}$ and $h_{2}$, $F^{*}$ must be
isomorphic to one component of $D05,D08,D09,D11$ depicted in Fig. 5.

If $\nabla(F^*)$ is a non-degenerate cyclic 7-edge-cut, then it can
be constructed from the trivial ones using the reverse operations of
$(O_{1}),(O_{2}),(O_{3})$ by Theorem \ref{obtainedbytrivialones}.
Note that there cannot be the subgraphs $L$ or $R$ in $F$. So in our
construction process we stop extending the cyclic 7-edge-cuts as
long as we encounter the two subgraphs. Denote by $5D$ the trivial
cyclic 5-edge-cut. In the following table we list the configurations
that arise when applying operations $(O_{1}^{-1})$, $(O_{2}^{-1})$
and $(O_{3}^{-1})$ and in Fig. 16 we give the corresponding
non-degenerated cyclic 7-edge-cuts.
\\\\

Table 1: Generating the non-degenerated cyclic 7-edge-cut. \vskip
0.2cm\scriptsize
\begin{tabular*}{\textwidth}{@{\extracolsep{\fill}}|c|c|c|c|c|c|c|c|c|c|c|c|c|c|c|}
\hline
cut&5D&6D01&6D02&6D03&D01&D02&D03&D04&D05&D06&D07&D08&D09&D11\\
\hline &&&D02&D06&&&&&&&&&&\\
$O_{1}^{-1}$&6D02&D01&D03&&---&---&---&---&---&---&---&---&---&---\\
&&&D04&D07&&&&&&&&&&\\
\hline &&&&&&D05&D05&D06&&D08&D09&D11&D11&D16\\
$O_{2}^{-1}$&---&---&6D03&6D04&D05&&&&D08&D09&&&&\\
&&&&&&D06&D06&D07&&D10&D10&D12&D13&D17\\
\hline
$O_{3}^{-1}$&5ND&6ND01&---&6ND02&---&---&---&---&ND01&---&---&ND02&---&ND03\\
\hline
\end{tabular*}
\normalsize

\begin{figure}[h]
\begin{center}
\includegraphics[scale=0.8]{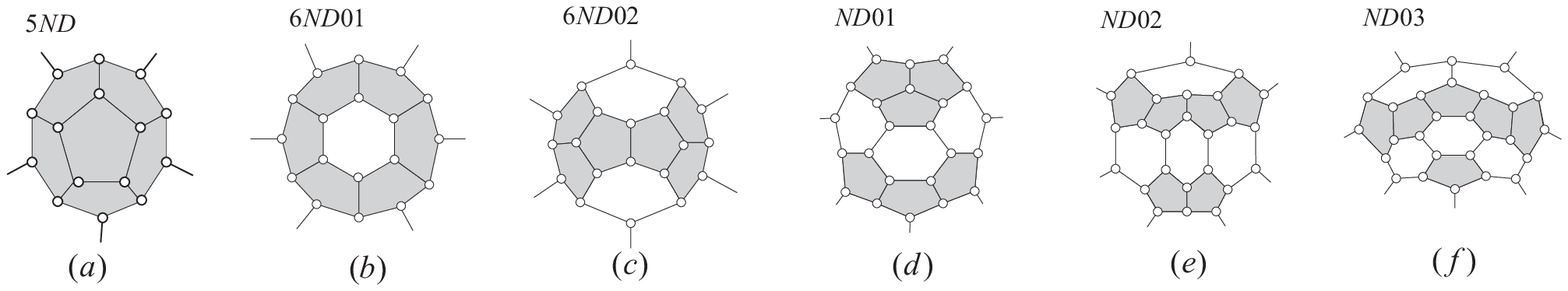}\\{Figure 16.  The non-degenerated cyclic 7-edge-cuts
when applying operations $(O_{1}^{-1})$, $(O_{2}^{-1})$ and
$(O_{3}^{-1})$.}
\end{center}
\end{figure}

Since the components of $5ND,6ND01,6ND02,ND01$ and $ND02$ (see Fig.
16 ($a$),($b$),($c$),

($d$),($e$)) contract the assumption, combining this with the
previous analysis, the components of $\nabla(F^*)$ contain a
subgraph isomorphic to the component of $ND03$ depicted in Fig.
16(f).

By the above analysis, we only need to show the following three
Claims hold in order to prove Lemma \ref{no7edgecut}.

\textbf{Claim 1:}  $F^{*}$ can not be isomorphic to the component of
$D05$ as illustrated in Fig.  5.

\begin{proof} To the contrary, $F^{*}$ is the component of $D05$. Let $v_{1},\cdots,v_{7}$ and $f_{1},\cdots,f_{7}$ be the vertices
and neighboring faces of $F^{*}$ as shown in Fig.  17(a). Then
$v_{1},\cdots,v_{7}$ are incident to $H$ or $A$. Moreover, there
exists at most one $E(h_{1},h_{2})$ or $E(A,A)$ or $E(A,H)$ edge or
one pentagonal factor-critical component by Ineqs. (5) and (6) and
Prop. \ref{5edgecut}. For convenience, we usually don't distinguish
$h_{1}$ and $h_{2}$.

\begin{figure}[h]
\begin{center}
\includegraphics[scale=0.7
]{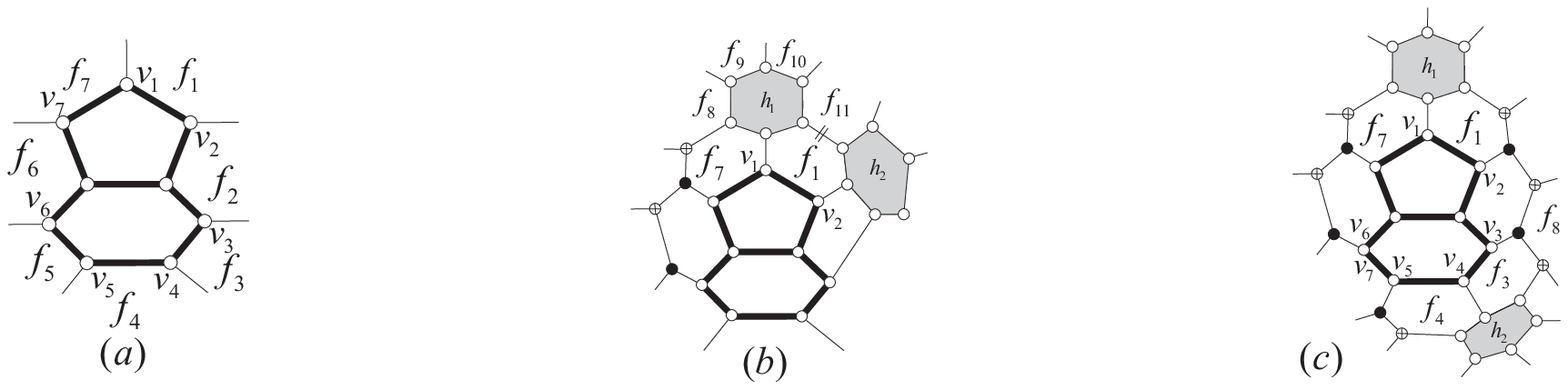}\\{Figure 17.  (a) The labellings of $F^{*}$, (b) the
case $v_{1},v_{2}$ incident to $h_{1},h_{2}$ (respectively), and (c)
the case $v_{1},v_{4}$ incident to $h_{1},h_{2}$ (respectively).}
\end{center}
\end{figure}

By Lemma \ref{both}, at least one of $v_{1},\cdots,v_{7}$ is
incident to $h_{1}$ or $h_{2}$. Firstly suppose $v_{1}$ is incident
to $h_{1}$. Then all of $v_{2},\cdots,v_{7}$ can not be incident to
$h_{1}$ by Lemma \ref{aforest}. Let $f_{8},f_{9},f_{10},f_{11}$ be
the four neighboring faces of $h_{1}$ different from $f_{1},f_{7}$
(see Fig.  17(b)). Then $v_{2}$ can not be incident to $h_{2}$,
otherwise, $h_{1}$ and $h_{2}$ are incident and all of
$f_{8},f_{9},f_{10},f_{11}$ are pentagons by Lemmas
\ref{intersectingboth} and \ref{characterizationh}(2), thus a
subgraph $L$ occurs in $F$ (impossible) (see Fig.  17(b)). So
$v_{2}$ is incident to $A$. By symmetry, $v_{7}$ is also incident to
$A$. If $v_{3}$ is incident to $h_{2}$, then $f_{2}$ contains an
$E(A,H)$ edge and there are no other contributing edges except for
this $E(A,H)$ edge and $E(F^{*})$. Now applying Lemma
\ref{characterizationf1} to $f_{1},\cdots,f_{7}$ we have
$p(H)+p(E(A,A))+p(\mathcal{D}^{*})\leq(4+4)+3=11$, contradicting
Ineq. (7). So $v_{3}$ is incident to $A$. Similarly for $v_{6}$. Now
we may assume $v_{4}$ is incident to $h_{2}$. Then all of
$f_{1},f_{3},f_{4},f_{7}$ are hexagons with the boundaries
$D^{*}D^{*}HHD_{0}A$ (see Fig.  17(c)), otherwise, $f_{1}$ or
$f_{3}$ or $f_{4}$ or $f_{7}$ contains an $E(A,H)$ edge and by Lemma
\ref{characterizationf1}
$p(H)+p(E(A,A))+p(\mathcal{D}^{*})\leq(4+4)+3=11$ (impossible).
Moreover, $p(V(h_{1}))\leq3,p(V(h_{2}))\leq3$. At this time at most
one $E(A,H)$ or $E(A,A)$ edge or one pentagonal factor-critical
component $F_{1}^{*}$ exists in $F$. However, since an $E(A,A)$ edge
gives rise to at most two pentagons and $p(V(F_{1}^{*}))\leq3$, we
obtain $p(H)+p(E(A,A))+p(\mathcal{D}^{*})\leq(3+3)+2+3=11$ no matter
which case occurs, also impossible by Ineq. (7). So $v_{4}$ is
incident to $A$. So does $v_{5}$ by symmetry. Now applying Corollary
\ref{both1} we can know $f_{10}$ or $f_{11}$ intersects both $h_{1}$
and $h_{2}$ with the boundary $HHD_{0}HHD_{0}$. Moreover, if there
is an $E(A,V(h_{2}))$ edge, then $f_{1},\cdots,f_{7}$ contain no
contributing edges. Thus $f_{1},f_{2},f_{6},f_{7}$ are hexagons and
$f_{3},f_{4},f_{5}$ are pentagons by Lemma \ref{characterizationf1}.
Hence $p(H)+p(E(A,A))+p(\mathcal{D}^{*})\leq(3+3)+4=10$
(impossible). If there is an $E(A,A)$ edge, say $e\in E(A,A)$, then
for $f_{11}$ intersecting both $h_{1}$ and $h_{2}$ we have all of
$f_{8},f_{9},f_{10}$ are pentagons with the boundaries
$HHD_{0}AD_{0}$ by Lemma \ref{intersectingboth} and $e$ must be
contained in $f_{15}$ (see Fig.  18(a) for the labellings of
$f_{12},\cdots,f_{16}$), otherwise, $f_{12}\cup f_{13}\cup
f_{14}\cup f_{15}$ forms a subgraph $L$ when $e$ belongs to $f_{16}$
and $f_{9}\cup f_{10}\cup f_{16}\cup f_{15}$ forms a subgraph $L$
when $e$ belongs to $f_{12}$ or $f_{13}$ or $f_{14}$. Now all of
$f_{12},f_{13},f_{14}$ are pentagons with the boundaries
$HHD_{0}AD_{0}$ by Lemma \ref{intersectingboth} and  $f_{2}$ is a
hexagon with the boundary $D^{*}D^{*}D^{*}AD_{0}A$ and $f_{3}$ is a
pentagon with the boundary $D^{*}D^{*}AD_{0}A$ by Lemma
\ref{characterizationf1}. Immediately the vertex $v$ (see Fig.
18(b)) belongs to three pentagons, contradicting the assumption.
Similarly for $f_{10}$ intersecting both $h_{1}$ and $h_{2}$ we have
$f_{8},f_{9},f_{11}$ are pentagons with the boundaries
$HHD_{0}AD_{0}$ and $e$ must belong to $f_{15}$. Again we have a
vertex $v$ belongs to three pentagons (see Fig.  18(c))
(impossible). If there is a pentagonal factor-critical component
$F_{1}^{*}$, then $F_{1}^{*}$ can not be incident to $h_{1}$,
otherwise, $p(V(h_{1}))\leq1,p(V(F_{1}^{*}))\leq2$ and
$p(H)+p(E(A,A))+p(\mathcal{D}^{*})\leq(3+1)+(4+2)=10$ (impossible).
Moreover, $f_{3}$ or $f_{4}$ or $f_{5}$ must contain an
$E(F_{1}^{*})$ edge to form a hexagonal neighboring face of
$F_{1}^{*}$ by Lemma \ref{characterizationf} and the fact that
$p(V(F_{1}^{*}))\leq3$. If $f_{3}$ or $f_{5}$ contains an
$E(F_{1}^{*})$ edge, say $f_{3}$, then $v_{8}$ is incident to
$h_{2}$ (see Fig.  18(d) for the labellings of $v_{8}$ and the
neighboring faces of $F_{1}^{*}$) since $F$ can not possess $R$ as
subgraph. Now by Lemmas \ref{characterizationf} and
\ref{characterizationf1} all of $f_{12},f_{15},f_{4},f_{5}$ are
pentagons with the boundaries $D^{*}D^{*}AD_{0}A$. Thus $f_{12}\cup
F_{1}^{*}\cup f_{15}\cup f_{4}$ forms a subgraph $L$ (see Fig.
18(d)), contradicting the assumption. If $f_{4}$ contains an
$E(F_{1}^{*})$ edge, then analogously $f_{12}\cup F_{1}^{*}\cup
f_{15}\cup f_{5}$ forms a subgraph $L$ (see Fig.  18(e)) (also
impossible). This contradiction means $v_{1}$ can not be incident to
$h_{1}$. Thus $v_{1}$ is incident to $A$.

\begin{figure}[h]
\begin{center}
\includegraphics[scale=0.7
]{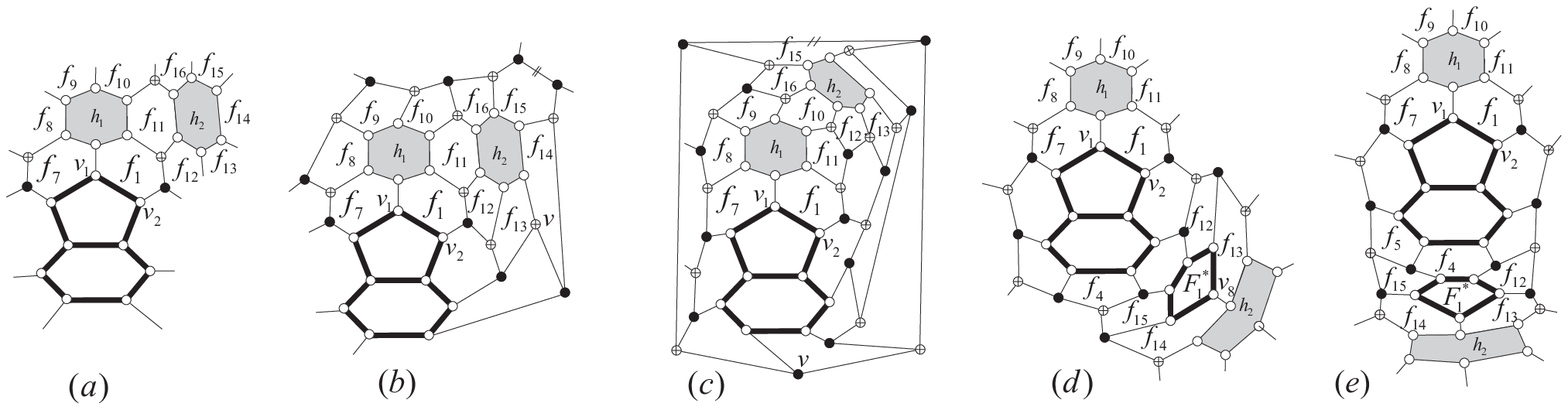}\\{Figure 18.  The case $v_{1}$ incident to $h_{1}$ and
all of $v_{2},\cdots,v_{7}$ incident to $A$.}
\end{center}
\end{figure}

Now suppose $v_{2}$ is incident to $h_{1}$, then $v_{3}$ is also
incident to $h_{1}$ to form the pentagon $f_{2}$ (see Fig.  19(a)),
otherwise, $f_{2}$ includes an $E(A,H)$ edge and there are no other
contributing edges by Ineqs. (5) and (6), which means $h_{2}$ is
incident to $F^{*}$ by Corollary \ref{atleastoneedge}. However, for
$v_{4}$ or $v_{5}$ or $v_{6}$ or $v_{7}$ incident to $h_{2}$ we'll
obtain $p(H)+p(E(A,A))+p(\mathcal{D}^{*})\leq(4+4)+3=11$ by Lemma
\ref{characterizationf1} (impossible). On the other hand, $f_{1}$ is
a hexagon with the boundary $D^{*}D^{*}HHD_{0}A$ since $v_{2}$ can
not belong to three pentagons by the assumption (see Fig.  19(a)).
Then if $v_{4}$ is incident to $h_{2}$, then $f_{3}$ contains an
$E(h_{1},h_{2})$ edge (see Fig.  19(a)) and there are no other
contributing edges, which means the neighboring faces of $h_{2}$
different from $f_{3},f_{4}$ are pentagons and they form a subgraph
$L$ in $F$ (impossible). So $v_{4}$ is incident to $A$. If $v_{5}$
is incident to $h_{2}$, then all of $f_{3},f_{4},f_{5}$ are hexagons
with the boundaries $D^{*}D^{*}HHD_{0}A$ (see Fig. 19(b)),
otherwise, $p(H)+p(E(A,A))+p(\mathcal{D}^{*})\leq(4+4)+3=11$
(impossible). Let $f_{8},f_{9},f_{10}$ and
$f_{11},f_{12},f_{13},f_{14}$ be the neighboring faces of $h_{1}$
and $h_{2}$, respectively (see Fig. 19(b)). Then at least one of
$f_{11},f_{12},f_{13},f_{14}$ is hexagonal in order to avoid the
occurrence of the subgraph $L$. On the other hand,
$f_{8},\cdots,f_{14}$ are pairwise different by Lemma \ref{aforest}.
Thus at least one of $f_{11},f_{12},f_{13},f_{14}$ contains a
contributing edge by Lemma \ref{intersectingboth}. For an
$E(h_{1},h_{2})$ edge contained in $f_{j}$ ($11\leq j\leq14$),
$f_{11}\cup f_{12}\cup f_{13}\cup f_{14}$ forms a subgraph $L$ by
Lemma \ref{characterizationh}(2)(impossible). For an $E(A,H)$ edge
contained in $f_{j}$ ($11\leq j\leq14$), we have
$p(H)+p(E(A,A))+p(\mathcal{D}^{*})\leq(2+4)+3=9$ (impossible). For
an $E(A,A)$ edge contained in $f_{j}$ ($11\leq j\leq14$), we have
$p(H)+p(E(A,A))+p(\mathcal{D}^{*})\leq(3+4)+1+3=11$ (impossible).
For an $E(F_{1}^{*})$ edge contained in $f_{j}$ ($11\leq j\leq14$)
we have $p(H)+p(E(A,A))+p(\mathcal{D}^{*})\leq(2+2)+(3+2)=9$
(impossible) when $F_{1}^{*}$ is incident to $h_{1}$ and
$p(H)+p(E(A,A))+p(\mathcal{D}^{*})\leq(2+4)+(3+3)=12$ when
$F_{1}^{*}$ is not incident to $h_{1}$, but now $f_{2}\in
\{P|V(P)\cap V(F^{*})\neq\emptyset\}\cap \{P|V(P)\cap
H\neq\emptyset\}$, contradicting that $|\mathcal{P}|=12$. So $v_{5}$
is incident to $A$. If $v_{6}$ is incident to $h_{2}$, then
similarly as the cases above, $v_{7}$ is also incident to $h_{2}$ to
form the pentagonal face $f_{6}$ and $f_{1},f_{3},f_{5},f_{7}$ are
all hexagons with the boundaries $D^{*}D^{*}HHD_{0}A$ (see Fig.
19(c)). Now for none or an $E(A,H)$ edge existence, we have
$p(H)+p(E(A,A))+p(\mathcal{D}^{*})\leq(4+4)+2=10$ (impossible). For
an $E(A,A)$ edge existence, say $e\in E(A,A)$, then $e$ can not be
contained in the neighboring faces of $h_{1}$ or $h_{2}$ or $F^{*}$,
otherwise, $p(H)+p(E(A,A))+p(\mathcal{D}^{*})\leq(4+4)+2+2-1=11$
(impossible). Thus the neighboring faces of $h_{1}$ ($h_{2}$)
different from $f_{1},f_{2},f_{3}$ ($f_{5},f_{6},f_{7}$) are
pentagons with the boundaries $HHD_{0}AD_{0}$ and $f_{4}$ is also a
pentagon with the boundary $D^{*}D^{*}AD_{0}A$ by Lemmas
\ref{intersectingboth} and \ref{characterizationf1}. Finally we
obtain the fullerene graph $F_{48}^{1}$ as shown in Fig.  19(d). For
an $E(F_{1}^{*})$ edge existence, we have a subgraph $R$ when
$F_{1}^{*}$ is neither incident to $h_{1}$ nor to $h_{2}$
(impossible) and the number of pentagons is at most (4+2)+(3+1)=10
when $F_{1}^{*}$ is incident to at least one of $h_{1},h_{2}$ by
Lemmas \ref{characterizationf}, \ref{characterizationf1},
contradicting that $|\mathcal{P}|=12$. This contradiction means
$v_{6}$ is incident to $A$. If $v_{7}$ is incident to $h_{2}$, then
$f_{6}$ contains an $E(A,H)$ edge and
$p(H)+p(E(A,A))+p(\mathcal{D}^{*})\leq(4+3)+3=10$ (impossible). Thus
$v_{7}$ is also incident to $A$. Let $f_{8},f_{9},f_{10}$ be the
neighboring faces of $h_{1}$ different form $f_{1},f_{2},f_{3}$ (see
Fig.  20(a)). Now we have a similar situation as the case $v_{1}$
incident to $h_{1}$ and for an $E(A,A)$ edge existence, we have the
fullerene graphs $F_{44}^{1},F_{44}^{2}$ as shown in Fig. 20(b),(c).
For $h_{2}$ incident to a pentagonal factor-critical component
$F_{1}^{*}$ we have the fullerene graph $F_{46}^{1}$ as depicted in
Fig.  20(d). But all of $F_{44}^{1},F_{44}^{2}$ and $F_{46}^{1}$ are
excluded in the assumption. Hence in the following we may assume
$v_{2}$ is incident to $A$. So does $v_{7}$ by symmetry.

\begin{figure}[h]
\begin{center}
\includegraphics[scale=0.7
]{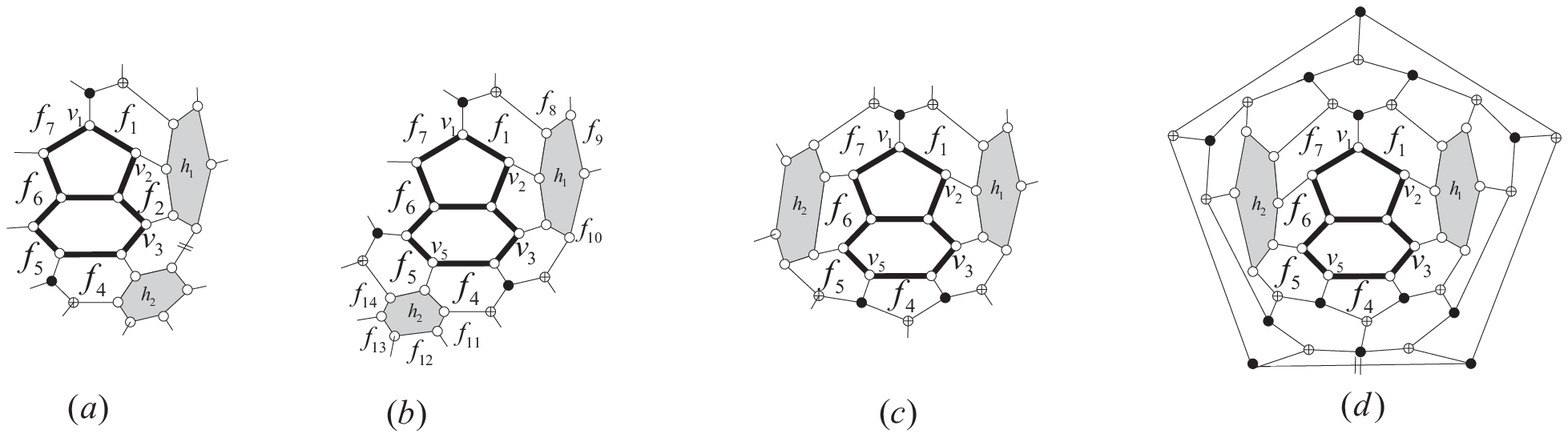}\\{Figure 19.  The case $v_{2},v_{3}$ incident to
$h_{1}$ and $v_{4}$ or $v_{5}$ or $v_{6},v_{7}$ incident to
$h_{2}$.}
\end{center}
\end{figure}

\begin{figure}[h]
\begin{center}
\includegraphics[scale=0.7
]{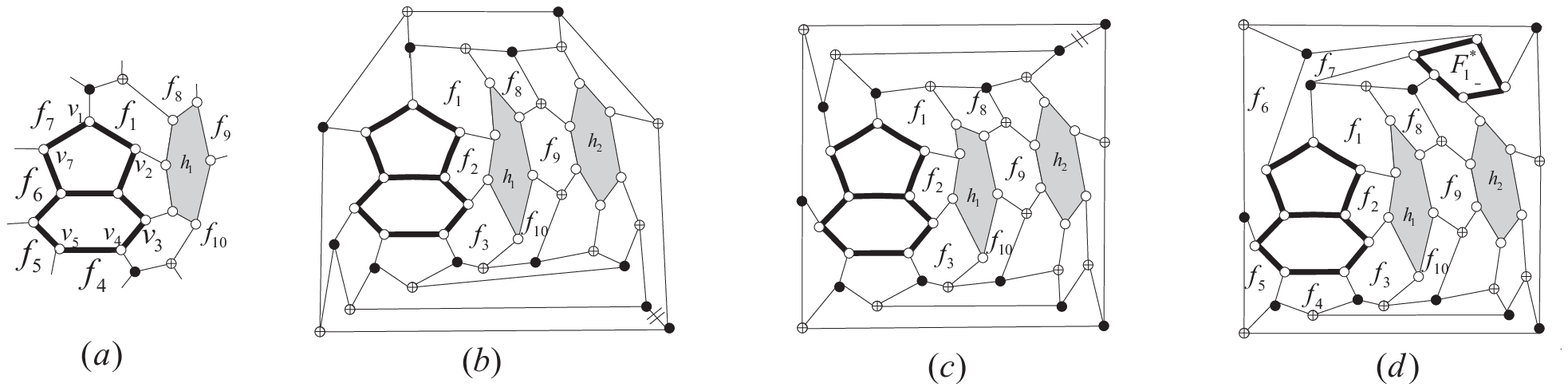}\\{Figure 20.  The case $v_{2},v_{3}$ incident to
$h_{1}$ and all of $v_{4},\cdots,v_{7}$ incident to $A$.}
\end{center}
\end{figure}

 Next suppose $v_{3}$ is incident to $h_{1}$, then $f_{2}$ contains
 an $E(A,H)$ edge and $p(H)+p(E(A,A))+p(\mathcal{D}^{*})\leq(4+3)+2=9$
 (impossible). Thus $v_{3}$ is also incident to $A$. So does $v_{6}$
 by symmetry. If $v_{4}$ is incident to $h_{1}$, then $v_{5}$ can not
 be incident to $h_{2}$, otherwise, $h_{1}$ and $h_{2}$ are incident
 and $v_{1}$ belongs to three pentagons by Lemma
 \ref{characterizationf1} (impossible). Hence
 $v_{5}$ is also incident to $A$. Using a similar discussion as the
 case $v_{1}$ incident to $h_{1}$ we can know this situation can not
 happen.
\end{proof}

\textbf{Claim 2:} $F^{*}$ can not be isomorphic to the component of
$D08$ or $D09$ or $D11$ as illustrated in Fig.  5.

\begin{proof} Also by contrary, suppose $F^{*}$ is isomorphic to the component of
$D08$ or $D09$ or $D11$. Let $v_{i}$ and $f_{i}$ ($1\leq i\leq7$) be
depicted in Fig. s 21(a), 22(a) and 22(d). Since $F$ does not
possess $L,R$ as subgraphs, $f_{7}$ is a hexagon when $F^{*}$ is
isomorphic to $D08$ (see Fig.  21(a)) and
$f_{2},f_{3},f_{5},f_{6},f_{7}$ are hexagons when $F^{*}$ is
isomorphic to $D09$ (see Fig.  22(a)) and $f_{1},f_{3},f_{5},f_{7}$
are hexagons when $F^{*}$ is isomorphic to $D11$ (see Fig.  22(d)).
By Lemma \ref{both} at least one of $v_{1},\cdots,v_{7}$ is incident
to $h_{1}$ or $h_{2}$. Then similarly as Claim 1, firstly we suppose
$v_{1}$ is incident to $h_{1}$ and we check that
$v_{2},\cdots,v_{7}$ are incident to $H$ or not. If $v_{1}$ is
finished, then we let $v_{2}$ be incident to $h_{1}$ and test that
$v_{3},\cdots,v_{7}$ are incident to $H$ or not. Execute the above
procedure ceaselessly until all of $v_{1},\cdots,v_{7}$ are incident
to $A$. Using the above checking procedure as Claim 1 finally we can
gain the fullerene graphs
$F_{46}^{2},F_{46}^{1},F_{48}^{2},F_{48}^{3}$ as shown in Fig.
21(b),(c),(d),(e) when $F^{*}$ is isomorphic to $D08$ and
$F_{46}^{4},F_{48}^{4}$ as shown in Fig.  22(b),(c) when $F^{*}$ is
isomorphic to $D09$. But these fullerene graphs are excluded in the
assumption.\end{proof}
\begin{figure}[h]
\begin{center}
\includegraphics[scale=0.7
]{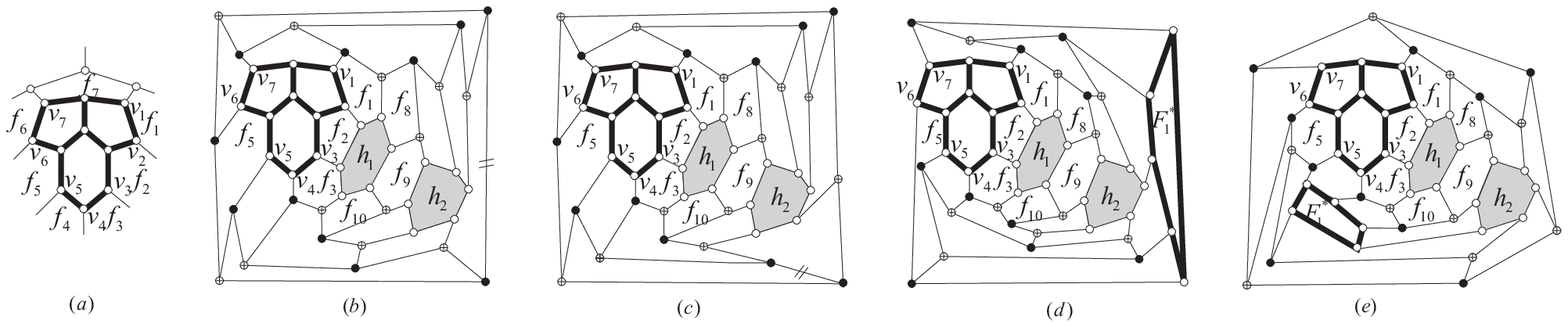}\\{Figure 21.  The case $F^{*}$ being isomorphic to the
component of $D08$.}
\end{center}
\end{figure}

\begin{figure}[h]
\begin{center}
\includegraphics[scale=0.7
]{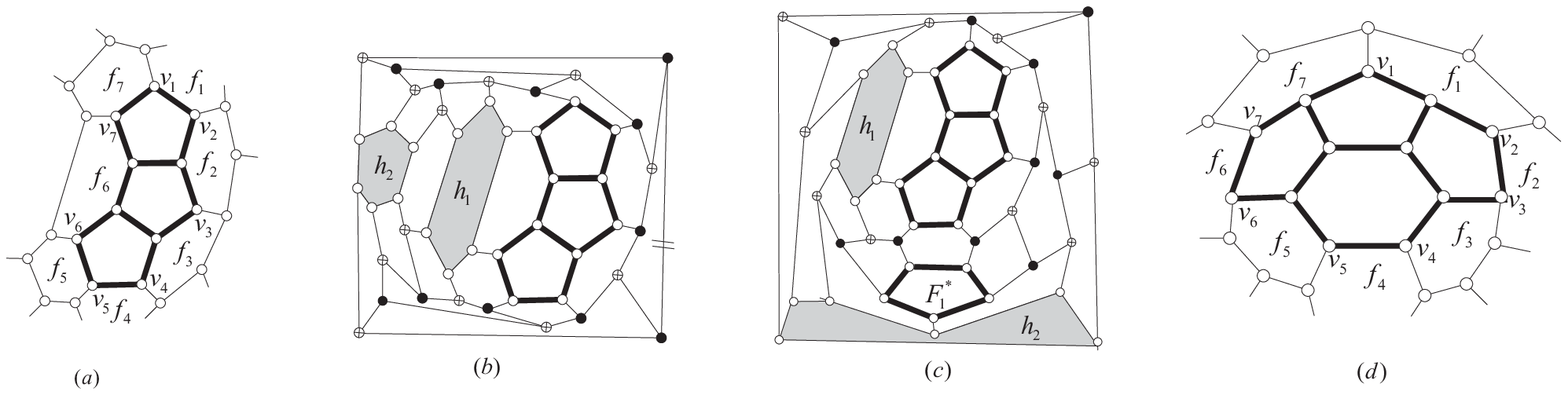}\\{Figure 22.  The case $F^{*}$ being isomorphic to the
component of $D09$ or $D11$.}
\end{center}
\end{figure}

\textbf{Claim 3:} $\overline{F^{*}}$ can not contain a subgraph
isomorphic to the component of $ND03$ as shown in Fig.  16 (f)

\begin{proof} For convenience, we write $G_{1}$ for the component of
$ND03$ shown in Fig.  16(f). By contrary suppose $\overline{F^{*}}$
contains the subgraph $G$. Denote by $f_{i}$ ($1\leq i\leq5$) and
$P_{i}$ ($1\leq i\leq6$) the hexagonal and pentagonal inner faces of
$G$ and $f_{j}$ ($6\leq j\leq12$) the six neighboring faces of $G$
depicted in Fig.  23(a). If $f_{1}$ equals $h_{1}$, then $f_{2}$
contains at least one contributing edges by Lemma
\ref{intersectingboth}. That is, there are no other contributing
edges except for the ones contained in $f_{2}$ and $E(F^{*})$ as
$s(\mathcal{D})\geq2$. In other words, $P_{6}$ must contain a vertex
in $H$ by (*). If $f_{3}$ is the hexagon $h_{2}$, then $f_{4}$
includes a contributing edge by Lemma \ref{intersectingboth},
contradicting Ineq. (5) (see Fig.  23(a) the case $f_{4}$ includes
an $E(A,H)$ edge). Similarly, none of $f_{4},f_{5},f_{6}$ and
$f_{7}$ is the hexagon $h_{2}$, which means $P_{6}$ contains no
vertex in $H$, also a contradiction. Thus $f_{1}$ does not equal
$h_{1}$. So does $f_{2}$ by symmetry. If $f_{3}$ equals $h_{1}$,
then similarly as the case $f_{1}$ equals $h_{1}$, $f_{4}$ contains
at least one contributing edge and $P_{1}$ must include a vertex in
$H$. However, no matter which face of $f_{1},f_{5},f_{8},f_{9}$ is
the hexagon $h_{2}$ we will obtain another contributing edge
different from the one contained in $f_{4}$, again contradicting
Ineq. (5). This contradiction means $f_{3}$ can not be the hexagon
$h_{1}$. So does $f_{5}$ by symmetry. If $f_{4}$ equals $h_{1}$,
then exactly one of $f_{3}$ and $f_{5}$ intersects both $h_{1}$ and
$h_{2}$, otherwise, each of $f_{3}$ and $f_{5}$ contains a
contributing edge by Lemma \ref{intersectingboth} and
$|E(A,H)|+|E(A,A)|+s(\mathcal{D})\geq4$ (impossible). Without loss of generality, suppose $f_{8}$ equals $h_{2}$ (see Fig.  23(b)),
then $f_{9}$ must be a hexagon, otherwise, $P_{1}\cup P_{2}\cup
P_{3}\cup f_{9}$ forms a subgraph $L$, contradicting the assumption.
Thus also $f_{9}$ contains a contributing edge different from the
one contained in $f_{3}$, which is impossible. So $f_{4}$ does not
equal $h_{1}$. Similarly, $f_{i}$ ($6\leq i\leq12$) cannot be the
hexagon $h_{1}$ or $h_{2}$. Since the only one possible non-trivial
factor-critical component of $F-H-A$ other than $F^{*}$ is a
pentagon by Prop. \ref{5edgecut}, whether $P_{i}$ ($1\leq i\leq6$)
is the pentagonal factor-critical component or not we will obtain
that $|E(A,A)|+s(\mathcal{D})\geq4$ by applying (*) to the six
pentagons $P_{i}$ ($1\leq i\leq6$), also a contradiction.\end{proof}

\begin{figure}[h]
\begin{center}
\includegraphics[scale=0.7
]{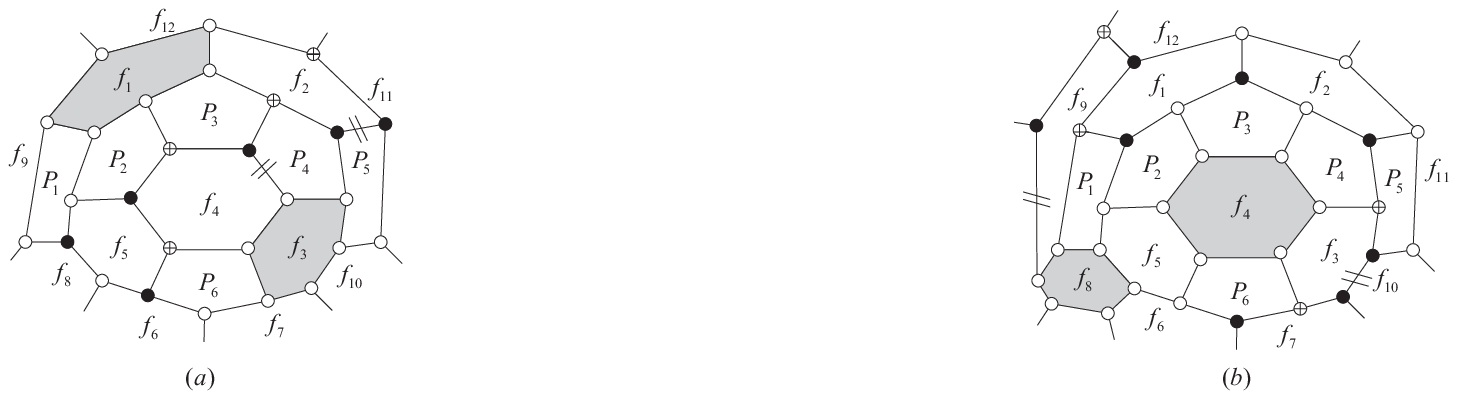}\\{Figure 23.  Illustration for Claim 3 in the proof of
Lemma \ref{no7edgecut}.}
\end{center}
\end{figure}

%%%%%%%%%%%%%%%%%%%%%%%%%%%%%%%%%%%%%%%%%%%%%%%%%%%%%%%%%%%%%%%%%%%%%%%%%%%%%%%%%%

\section{Proof of Lemma \ref{no5edgecut}}

Since every $F^{*}\in \mathcal{D}$ with $|\nabla(F^*)|=5$ is a
pentagon by Prop. \ref{5edgecut} and $1\leq p(V(F^{*}))\leq3$ by
Ineq. (8), we indicate the following three Claims hold to prove
Lemma \ref{no5edgecut}.

\textbf{Claim 1:}  There is no a pentagonal component $F^{*}\in
\mathcal{D}^{*}$ such that $p(V(F^{*}))=3$.

\begin{proof} Suppose to the contrary that there exists one such pentagonal factor-critical
component $F^{*}$. Denote by $v_{1}v_{2}v_{3}v_{4}v_{5}v_{1}$ the
boundary of $F^{*}$ along the clockwise direction and $v_{i}'$ the
neighbor of $v_{i}$ ($i=1,2,3,4,5$) not in $F^{*}$. Let
$f_{1},f_{2},\cdots,f_{5}$ be the five neighboring faces of $F^{*}$
along the edges $v_{1}v_{2},v_{2}v_{3},\cdots,v_{5}v_{1}$,
respectively. Lemma\ref{lessthan1} guarantees at most one of
$v_{1},v_{2},\cdots,v_{5}$ is incident to $h_{j}$ for some $j\in
\{1,2\}$. Note if there exists another
$F_{1}^{*}\in\mathcal{D}^{*}$, then $F_{1}^{*}$ is also a pentagon
by Lemma \ref{no7edgecut} and the fact $s(\mathcal{D})\leq3$. We
always do not distinguish $h_{1}$ and $h_{2}$.

As $p(V(F^{*}))=3$, two of $f_{1},f_{2},\cdots,f_{5}$ are pentagons.
Then the two pentagons are not adjacent by the assumption. By
symmetry, we can suppose $f_{1}$ and $f_{4}$ are the two pentagons.
Denote by $f_{6},f_{7},\cdots,f_{13}$ the faces and
$v_{6},v_{7},\cdots,v_{13}$ the vertices of $F$ as shown in Fig.
24(a). Then $f_{6},f_{12}$ are hexagons in order to prevent the
forbidden subgraph $L$ occurring in $F$.

\begin{figure}[h]
\begin{center}
\includegraphics[scale=0.7
]{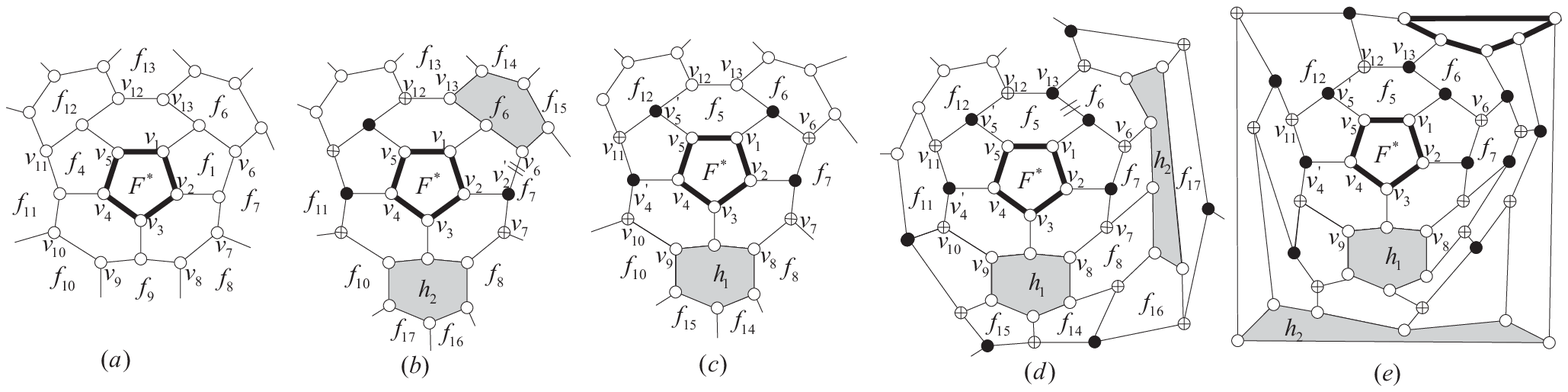}\\{Figure 24.  Illustration for Claim 1 in the proof of
Lemma \ref{no5edgecut}.}
\end{center}
\end{figure}

If $v_{1}$ is incident to $h_{1}$, then $f_{1}$ contains an $E(A,H)$
edge $v_{2}'v_{6}$ (see Fig.  24(b)). If $v_{3}$ is incident to
$h_{2}$, then $v_{4}',v_{5}'\in A$ and
$v_{7},v_{10},v_{11},v_{12}\in D_{0}$ by Ineqs. (5) and (6). Let
$f_{13},f_{14},f_{15}$ and $f_{16},f_{17}$ be the neighboring faces
of $h_{1}$ and $h_{2}$ (respectively) as shown in Fig. 24(b). It's
easy to see the neighboring faces of $h_{1}$ and $h_{2}$ are
pairwise different. Thus at least one of
$f_{8},f_{10},f_{16},f_{17}$ contains a contributing edge by Lemma
\ref{intersectingboth} in order to prevent the forbidden subgraph
$L$. That is, $p(V(h_{2}))\leq3$. If this additional contributing
edge belongs to $E(h_{1},h_{2})$ or $E(A,H)$, then
$p(H)+p(E(A,A))+p(D^{*})\leq(4+3)+3=10$ (impossible). If it belongs
to $E(A,A)$, then $p(H)+p(E(A,A))+p(D^{*})\leq(4+3)+3+1=11$ (also
impossible). If it belongs to $E(F_{1}^{*})$, then $F_{1}^{*}$ must
be incident to both $h_{1}$ and $h_{2}$ as $p(V(F_{1}^{*}))\leq3$,
which means
$p(V(h_{2}))\leq2,p(V(h_{1}))\leq3,p(V(F_{1}^{*}))\leq2$. Hence
$p(H)+p(E(A,A))+p(D^{*})\leq(3+2)+(3+2)=10$ (impossible). This
contradiction implies $v_{3}$ is incident to $A$. If $v_{4}$ is
incident to $h_{2}$, then $f_{4}$ contains an $E(A,H)$ edge
$v_{11}v_{5}'$ and $f_{2}$ also contains a contributing edge
belonging to $E(A,A)$ or $E(F_{1}^{*})$ by Lemma
\ref{characterizationf}. Thus
$|E(A,H)|+|E(A,A)|+s(\mathcal{D})\geq4$, contradicting Ineq. (5). So
$v_{4}$ is also incident to $A$. If $v_{5}$ is incident to $h_{2}$,
then $h_{1}$ and $h_{2}$ are incident and each of $f_{2},f_{3}$
contains a contributing edge, which is impossible by Ineq. (6). Now
all of $v_{2},\cdots,v_{5}$ are incident to $A$ and we once again
obtain that $|E(A,H)|+|E(A,A)|+s(\mathcal{D})\geq4$ (impossible).
Thus $v_{1}$ is incident to $A$. So does $v_{5}$ by symmetry.

If $v_{2}$ is incident to $h_{1}$, then $f_{1}$ contains an $E(A,H)$
edge $v_{1}'v_{6}$ and we have a similar situation as the case
$v_{1}$ incident to $h_{1}$. The same analysis will deduce that for
$v_{3}$ or $v_{4}$ incident to $h_{2}$ or $A$ we'll obtain
$p(H)+p(E(A,A))+p(D^{*})\leq(3+2)+(3+2)=10$ (impossible) or
$|E(A,H)|+|E(A,A)|+s(\mathcal{D})\geq4$ (also impossible). Thus
$v_{2}$ is incident to $A$. Similarly for $v_{4}$.

If $v_{3}$ is incident to $h_{1}$, then
$v_{6},v_{7},v_{10},v_{11}\in D_{0}$ and $f_{5}$ contains an
$E(A,A)$ or $E(F_{1}^{*})$ edge by Lemma \ref{characterizationf}.
Denote by $f_{14},f_{15}$ the neighboring faces of $h_{1}$ different
form $f_{2},f_{3},f_{8},f_{10}$ (see Fig.  24(c)). Firstly suppose
$v_{1}'v_{13}\in E(A,A)$ (see Fig.  24(d)). By Ineqs. (5) and (6)
there is at most one another $E(h_{1},h_{2})$ or $E(A,H)$ or
$E(A,A)$ edge or one $F_{1}^{*}\in\mathcal{D}$ with
$|\nabla(F_{1}^{*})|=5$. For none or one $E(h_{1},h_{2})$ or one
$E(A,H)$ edge existence, we have
$p(H)+p(E(A,A))+p(D^{*})\leq(4+4)+3=11$ (impossible). For one
additional $E(A,A)$ edge existence, say $e\in E(A,A)$, we obtain
$p(\{e\})=2$, otherwise, $p(H)+p(E(A,A))+p(D^{*})\leq(3+4)+1+3=11$
(impossible). Moreover, $f_{6}$ intersects $h_{2}$ with the boundary
$HHD_{0}AAD_{0}$ by Observation \ref{characterizationm2} and $e$ can
not be contained in the neighboring faces of $h_{1}$ and $h_{2}$ by
Lemma \ref{characterizationh}(1)(3). Thus one of
$f_{8},f_{10},f_{14},f_{15}$ must intersect both $h_{1}$ and $h_{2}$
with the boundary $HHD_{0}HHD_{0}$ and the remaining neighboring
faces of $h_{1}$ and $h_{2}$ are pentagons with the boundaries
$HHD_{0}AD_{0}$ in order to avoid the occurring the forbidden
subgraph $L$. From the above analysis we can construct the fullerene
graph. It's easy to see the four neighboring faces
($f_{14},\cdots,f_{17}$) of $h_{1}$ and $h_{2}$ form a subgraph $L$
(see Fig.  24(d)) (impossible). For $F_{1}^{*}$ existence, $f_{6}$
must contain an $E(F_{1}^{*})$ edge and $F_{1}^{*}$ must be incident
to $h_{2}$ but not $h_{1}$, otherwise, either it happens a subgraph
$R$ in $F$ or $p(H)+p(E(A,A))+p(D^{*})\leq11$, both of which are
impossible by the assumption and Ineq. (7). Hence the positions of
$h_{2}$ and $F_{1}^{*}$ are known and we have the fullerene graph
$F_{46}^{3}$ as shown in Fig.  24(e), which is excluded in the
assumption. Next we assume $v_{12},v_{13}\in V(F_{2}^{*})$. Then one
of $f_{6},f_{12}$ intersects $h_{2}$ and the remaining one contains
an $E(A,A)$ edge by Lemma \ref{characterizationf}. Without loss of generality, suppose $f_{6}$ intersects $h_{2}$. Again four
neighboring faces of $h_{1}$ and $h_{2}$ form a subgraph $L$
(impossible). This contradiction means $v_{3}$ is incident to $A$.
Thus all of $v_{1},\cdots,v_{5}$ are incident to $A$ and by Lemma
\ref{characterizationf} $|E(A,A)|+s(\mathcal{D})\geq4$ (impossible).
\end{proof}

To make a summary, we can see $p(V(F^{*}))\leq2$ for any
$F^{*}\in\mathcal{D}$ with $|\nabla(F^{*})|=5$.

\textbf{Claim 2:}  There is no a pentagonal component $F^{*}\in
\mathcal{D}^{*}$ such that $p(V(F^{*}))=2$.

\begin{proof} By contrary such a component $F^{*}$ exists. Label the
boundary of $F^{*}$ and its neighboring faces as shown in Fig.
25(a). We also have if there exists another
$F_{1}^{*}\in\mathcal{D}^{*}$, then $F_{1}^{*}$ is also a pentagon.
As $p(V(F^{*}))=2$, we can suppose $f_{1}$ is pentagonal. Similarly
as Claim 1, for $v_{2}$ incident to $h_{1}$ we have $f_{1}$ contains
an $E(A,H)$ edge $v_{1}'v_{6}$ (see Fig.  25(b)) and $v_{3}$ cannot
be incident to $h_{2}$, otherwise, $h_{1}$ and $h_{2}$ are incident
and $p(H)+p(E(A,A))+p(\mathcal{D}^{*})\leq(4+4)+2=10$ (impossible).
Moreover, $v_{4}$($v_{5}$) also cannot be incident to $h_{2}$,
otherwise, $p(V(h_{2}))\leq3$ and $f_{5}$($f_{3}$) contains an
$E(A,A)$ or $E(F_{1}^{*})$ edge by Lemma \ref{characterizationf},
but $p(H)+p(E(A,A))+p(\mathcal{D}^{*})\leq(4+3)+2+2=11$
(impossible). So all of $v_{3}',v_{4}',v_{5}'$ belong to $A$ and
$|E(A,H)|+|E(A,A)|+s(\mathcal{D})\geq5$ (impossible). Hence
$v_{2}'\in A$. So does $v_{1}'$ by symmetry and $v_{6}\in D_{0}$.
For $v_{3}$ incident to $h_{1}$, then $v_{4}$ cannot be incident to
$h_{2}$, otherwise, $h_{1},h_{2}$ are incident and
$p(V(h_{1}))\leq3,p(V(h_{2}))\leq3$ and $f_{5}$ contains an $E(A,A)$
or $E(F_{1}^{*})$ edge. However, no matter which case occurs we'll
have $p(H)+p(E(A,A))+p(\mathcal{D}^{*})\leq(3+3)+2+2=10$
(impossible). If $v_{5}$ is incident to $h_{2}$, then also
$p(V(h_{1}))\leq3,p(V(h_{2}))\leq3$. Denote by
$f_{6},f_{7},f_{8},f_{9}$($f_{10},f_{11},f_{12},f_{13}$) the four
neighboring faces of $h_{1}$($h_{2}$) different from
$f_{2},f_{3}$($f_{4},f_{5}$)(see Fig.  25(c)). Then the eight faces
$f_{6},f_{7},\cdots,f_{13}$ are pairwise different by Lemma
\ref{equalemptyset}. Hence in order to prevent the occurring of the
forbidden subgraph $L$, at least one of $f_{6},f_{7},f_{8},f_{9}$
contains a contributing edge by Lemma \ref{intersectingboth}.
Similarly for $f_{10},f_{11},f_{12},f_{13}$. If these contributing
edges belong to $E(A,H)$ or $E(A,A)$, then at most two such edges
exist in $F$ and $p(H)+p(E(A,A))+p(D^{*})\leq(3+3)+2+2=10$ (Note an
$E(A,A)$ edge gives rise to one hexagon by
Lemma\ref{characterizationh}(1)(3)) (impossible). So there must
exist additional $E(\mathcal{D}^{*})$ edges (other than $E(F^{*})$),
which means at least one of $h_{1},h_{2}$(say $h_{1}$) is incident
to another pentagonal factor-critical component, say $F_{1}^{*}$
(see Fig.  25(c)). Thus $p(V(h_{1}))\leq2$ since the two common
neighboring faces of $h_{1}$ and $F_{1}^{*}$ are hexagons by
Lemma\ref{characterizationf}. On the other hand, if the contributing
edge contained in $f_{10}$ or $f_{11}$ or $f_{12}$ or $f_{13}$
belongs to $E(A,A)$ or $E(A,H)$, then we once again obtain that
$p(H)+p(E(A,A))+p(D^{*})\leq(2+3)+1+(2+2)=10$,
 (impossible). If the contributing edge
 contained in $f_{j}$ for some $j\in \{10,11,12,13\}$ belongs to $F_{2}^{*}\in\mathcal{D}$ with $|\nabla(F_{2}^{*})|=5$, then
 similarly as the case above we have $p(V(h_{2}))\leq2$. If
 $F_{2}^{*}=F_{1}^{*}$, then
 $p(H)+p(E(A,A))+p(D^{*})\leq(2+2)+2+(2+2)=10$ (impossible). If $F_{2}^{*}\neq F_{1}^{*}$, then $p(E(A,A))=0$ and $p(H)+p(E(A,A))+p(D^{*})\leq(2+2)+(2+2+2)=10$
 (impossible). Hence $v_{5}$ is also incident to $A$. Now both of
 $f_{4},f_{5}$ contain an $E(A,A)$ or $E(F_{1}^{*})$ edge by Lemma
 \ref{characterizationf}(1). However, using a similar discussion as
 the cases above we can know no matter which case occurs we'll obtain
 $p(H)+p(E(A,A))+p(D^{*})\leq11$ (impossible). Thus $v_{3}$
 is incident to $A$. So does $v_{5}$ by symmetry. Now both of $f_{2},f_{5}$ contain an $E(A,A)$ or $E(F_{1}^{*})$
 edge by Lemma \ref{characterizationf} and for $v_{4}$
 incident to $h_{1}$ we obtain $p(H)+p(E(A,A))+p(D^{*})\leq11$
 (impossible) and for $v_{4}$ incident to $A$ we can gain $|E(A,A)|+s(\mathcal{D})\geq4$ (also impossible).\end{proof}

\begin{figure}[h]
\begin{center}
\includegraphics[scale=0.7
]{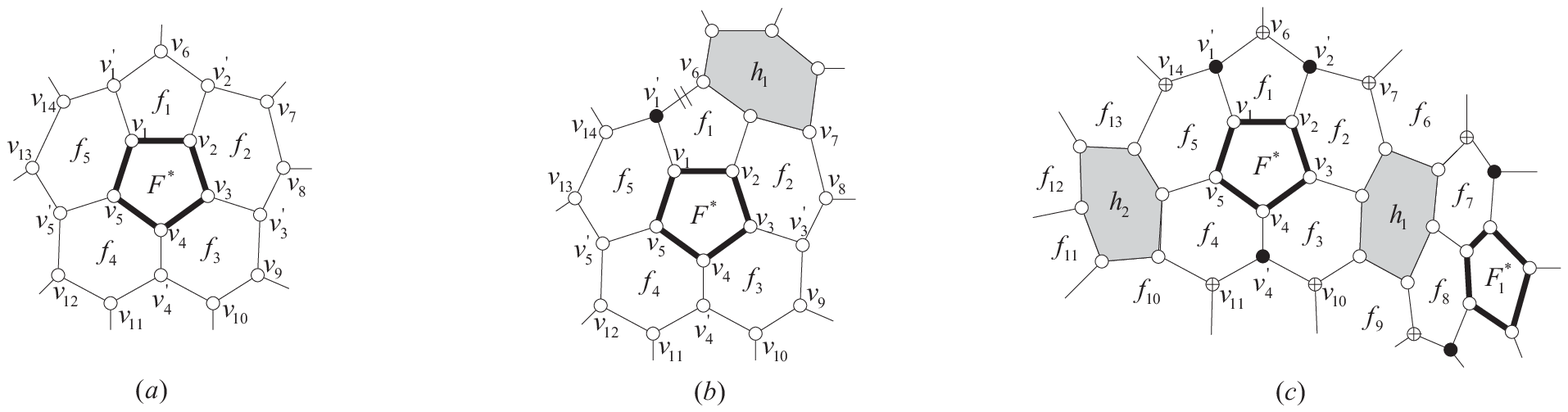}\\{Figure 25. Illustration for Claim 2 in the proof of
Lemma \ref{no5edgecut}.}
\end{center}
\end{figure}

From the above analysis, we have $p(V(F^{*}))\leq1$ for any
 $F^{*}\in\mathcal{D}$ with $|\nabla(F^{*})|=5$.

\textbf{Claim 3:}  There exists no a pentagonal component $F^{*}\in
\mathcal{D}^{*}$ such that $p(V(F^{*}))=1$.

\begin{proof} Suppose there is one such component $F^{*}$. Also denote by $v_{1}v_{2}v_{3}v_{4}v_{5}v_{1}$
the boundary of $F^{*}$ and $f_{1},\cdots,f_{5}$ the neighboring
faces of $F^{*}$ along the edges
$v_{1}v_{2},v_{2}v_{3},\cdots,v_{5}v_{1}$, respectively. Since
$p(V(F^{*}))=1$,  $f_{i}$ is hexagonal for all $i\in
\{1,2,\cdots,5\}$. Similarly as Claims 1,2, for at most one of
$v_{1},v_{2},\cdots,v_{5}$ incident to $h_{1}$, we obtain
$|E(A,A)|+s(\mathcal{D})\geq4$ (impossible). For two of
$v_{1},v_{2},\cdots,v_{5}$ incident to $h_{1}$ and $h_{2}$
(respectively), we always have
$|E(A,A)|+s(\mathcal{D})+|E(V(h_{1}),V(h_{2}))|\geq4$ when the two
vertices incident to $h_{1}$ and $h_{2}$ are adjacent on
$\partial(F^{*})$ or $p(H)+p(E(A,A))+p(D^{*})\leq11$ when the two
vertices incident to $h_{1}$ and $h_{2}$ are not adjacent on
$\partial(F^{*})$, which are impossible.
\end{proof}

%%%%%%%%%%%%%%%%%%%%%%%%%%%%%%%%%%%%%%%%%%%%%%%%%%%%%%%%%%%%%%%%%%%%%%%%%%%%%%%%%%
%%References----------------------------------------------------

\end{document}